\documentclass[11pt,a4paper,center,notitlepage]{article}
\usepackage{tabularx}
\usepackage[english]{babel}
\usepackage{color}
\usepackage{textcomp,multicol,enumerate,amsmath,amssymb,amsthm,eufrak,latexsym,makeidx}
\usepackage{subfig}
\usepackage{graphicx,epstopdf}
\usepackage[usenames,dvipsnames,svgnames,table]{xcolor}
\usepackage[colorlinks=true,
            linkcolor=blue,
            urlcolor=gray,
            citecolor=magenta]{hyperref}

\newcommand{\vertiii}[1]{{\left\vert\kern-0.25ex\left\vert\kern-0.25ex\left\vert #1
    \right\vert\kern-0.25ex\right\vert\kern-0.25ex\right\vert}}
\allowdisplaybreaks
\numberwithin{equation}{section}
\addto{\captionsenglish}{  
}
\newtheorem{lemma}{Lemma}

\newtheorem{definition}{Definition}
\newtheorem{proposition}{Proposition}
\newtheorem{theorem}{Theorem}

\newtheorem{remark}{Remark}

\input{tcilatex}
\begin{document}

\title{Pointwise gradient estimates in multi-dimensional slow diffusion
equations with a singular quenching term.}
\author{Nguyen Anh Dao\thanks{%
Institute of Mathematical Sciences, ShanghaiTech University, China.\newline
Email: \texttt{dnanh@shanghaitech.edu.cn}} \and Jesus Ildefonso D\'{\i}az%
\thanks{%
Instituto de Matem\'{a}tica Interdisciplinar, Universidad Complutense de
Madrid, 28040 Madrid Spain.\newline
Email: \texttt{jidiaz@ucm.es}} \and Quan Ba Hong Nguyen\thanks{%
UFR Math\'ematiques, Institut de Recherche Math\'ematique de Rennes (IRMAR),
Universit\'e de Rennes 1, Beaulieu, 35042 Rennes, France.\newline
Email: \texttt{nguyenquanbahong@gmail.com}} }
\maketitle

\begin{abstract}
We consider the high-dimensional equation, ${\partial _{t}}u-\Delta {u^{m}}+{%
u^{-\beta }}{\chi _{\left\{ {u>0}\right\} }}=0$, extending the mathematical
treatment made on 1992 by B. Kawohl and R. Kersner for the one-dimensional
case. Besides the existence of a very weak solution $u\in \mathcal{C}\left( %
\left[ 0,T\right] ;L_{\delta }^{1}\left( \Omega \right) \right) $, with ${%
u^{-\beta }}{\chi _{\left\{ {u>0}\right\} }}\in {L^{1}}\left( {\left( {0,T}%
\right) \times \Omega }\right) $, $\delta \left( x\right) =d\left(
x,\partial \Omega \right) $, we prove some pointwise gradient estimates \
for a certain range of the dimension $N$, $m\geq 1$ and $\beta \in \left(
0,m\right) $, mainly when the absorption dominates over the diffusion ($%
1\leq m<2+\beta $). In particular, a new kind of universal gradient estimate
is proved when $m+\beta \leq 2$. Several qualitative properties (such as the
finite time quenching phenomena and the finite speed of propagation) and the
study of the Cauchy problem are also considered.
\end{abstract}

\tableofcontents



\begin{center}
\emph{Dedicated to Laurent V\'{e}ron in occasion of his 70th birthday}
\end{center}

\section{Introduction and main results}

\subsection{Introduction}

The main goal of this paper is to extend to the high-dimensional case, the
1992 mathematical treatment made by B. Kawohl and R. Kersner \cite%
{Kawohl1992a} for a one-dimensional degenerate diffusion equation with a
singular absorption term. More precisely, we will study nonnegative
solutions of the following possibly degenerate reaction-diffusion
multi-dimensional problem

\begin{equation}
\left\{ 
\begin{split}
& {\partial _{t}}u-\Delta {u^{m}}+{u^{-\beta }}{\chi _{\left\{ {u>0}\right\}
}}=0,\mbox{ in }\left( {0,\infty }\right) \times \Omega , \\
& u^{m}=0,\mbox{ on }\left( {0,\infty }\right) \times \partial \Omega , \\
& u\left( {0,x}\right) ={u_{0}}\left( x\right) ,\mbox{ in }\Omega ,
\end{split}%
\right.   \tag{P}  \label{P}
\end{equation}%
where $\Omega $ is an open regular bounded domain of $\mathbb{R}^{N}$ (for
instance with $\partial \Omega $ of class $C^{1,\alpha }$, for some $\alpha
\in (0,1]$), $N\geq 1$, $m\geq 1$ ($m>1$ corresponds to a typical slow
diffusion) and mainly $\beta \in \left( 0,m\right) $ (some remarks will be
made on the case $\beta \geq m$ at the end of this paper). The case of the
whole space, $\Omega =\mathbb{R}^{N}$, will be treated separately. Here ${%
\chi _{\left\{ {u>0}\right\} }}$ denotes the characteristic function of the
set of points $\left( t,x\right) $ where $u\left( t,x\right) >0$, i.e.: 
\begin{equation*}
{\chi _{\left\{ {u>0}\right\} }}\left( {t,x}\right) :=\left\{ 
\begin{split}
& 1,\mbox{ if }u\left( {t,x}\right) >0, \\
& 0,\mbox{ if }u\left( {t,x}\right) =0.
\end{split}%
\right. 
\end{equation*}%
Note that the absorption term ${u^{-\beta }}{\chi _{\left\{ {u>0}\right\} }}$
becomes singular (and the diffusion becomes degenerate if $m>1$) when $u=0$,
and that by this normalization we have that ${u^{-\beta }}{\chi _{\left\{ {%
u>0}\right\} }}\left( t,x\right) =0$ if $u\left( t,x\right) =0$. Notice that
the boundary condition implies an automatic permanent singularity on the
boundary $\partial \Omega $, in contrast to other related problems in which
the singularity is permanently excluded of the boundary 
\begin{equation}
\left\{ 
\begin{split}
& {\partial _{t}}u-\Delta {u^{m}}+{u^{-\beta }}{\chi _{\left\{ {u>0}\right\}
}}=0,\mbox{ in }\left( {0,\infty }\right) \times \Omega , \\
& u^{m}=1,\mbox{ on }\left( {0,\infty }\right) \times \partial \Omega , \\
& u\left( {0,x}\right) ={u_{0}}\left( x\right) ,\mbox{ in }\Omega .
\end{split}%
\right.   \tag{P(1)}  \label{P(1)}
\end{equation}%
Notice also that the change of unknown $v=1-u^{m}$, with $u$ solution of %
\eqref{P(1)}, in the semilinear case ($m=1$), for instance, leads to the
formulation 
\begin{equation}
\left\{ 
\begin{split}
& {\partial _{t}}v-\Delta v=\frac{{{\chi _{\left\{ {v<1}\right\} }}}}{{{{%
\left( {1-v}\right) }^{\beta }}}},\mbox{ in }\left( {0,\infty }\right)
\times \Omega , \\
& v=0,\mbox{ on }\left( {0,\infty }\right) \times \partial \Omega , \\
& v\left( {0,x}\right) =1-{u_{0}}\left( x\right) ,\mbox{ in }\Omega .
\end{split}%
\right.   \label{P(Kawarada)}
\end{equation}%
In this way, the study of the associated Cauchy problem 
\begin{equation}
\left\{ 
\begin{split}
& {\partial _{t}}u-\Delta {u^{m}}+{u^{-\beta }}{\chi _{\left\{ {u>0}\right\}
}}=0,\mbox{ in }\left( {0,\infty }\right) \times {\mathbb{R}^{N}}, \\
& u\left( {0,x}\right) ={u_{0}}\left( x\right) ,\mbox{ in }{\mathbb{R}^{N}},
\end{split}%
\right.   \tag{CP}  \label{PC}
\end{equation}%
can be regarded from two different points of view according to the
assumptions made on the asymptotic behavior of the initial datum when $%
\left\vert x\right\vert \rightarrow +\infty $. The case ${u_{0}}\left(
x\right) \searrow 0$, as $\left\vert x\right\vert \rightarrow +\infty $, can
be considered as a limit of problems of the type \eqref{P}, and the case in
which $u_{0}\left( x\right) $ is growing with $\left\vert x\right\vert $, as 
$\left\vert x\right\vert \rightarrow +\infty $, corresponds to a limit of
problems of the type \eqref{P(1)} (see, e.g., \cite{Guo 2007}). Our main
goal in this paper is to analyze problems of the type \eqref{P} and %
\eqref{PC} when ${u_{0}}\left( x\right) \searrow 0$ as $\left\vert
x\right\vert \rightarrow +\infty $.

The literature on this type of problems increased very quickly in the last
decades. Problem \eqref{P} (and \eqref{P(1)}) was regarded as the limit case
of the regularized Langmuir-Hinshelwood model in chemical catalyst kinetics
(see \cite{Aris, Diaz1985, Diaz1987} for the elliptic case and \cite%
{Bandle1986, Phillips1987} for the parabolic equation). Some regularized
singular absorption terms also arise in some models in enzyme kinetics (\cite%
{Banks1975}). See also many other references in the survey \cite%
{Hernandez2006}.

As mentioned before, what makes specially interesting equations like %
\eqref{P} is the fact that the solutions may raise to a free boundary
defined as $\partial \left\{ \left( t,x\right) ;u\left( t,x\right)
>0\right\} $. In some contexts, problem \eqref{P(1)} was denoted as \textit{%
the} \textit{quenching problem}. It was soon pointed out the appearance of a
blow-up time for $\partial _{t}u$ at the first time $T_{c}>0$ in which $%
u\left( T_{c},x\right) =0$ at some point $x\in \Omega $ (see, e.g., \cite%
{Kawarada1974/75, Levine1993, Phillips1987}). More recently, parabolic
problems with a singular absorption term of this type have been investigated
by many authors (see, e.g., \cite{Dao2017, Dao2019b, Dao2016b, Davila2005,
Kawohl1996, Levine1993, Phillips1987, Winkler2007}, and references therein).
Concerning the associate semilinear Cauchy problem we mention the papers 
\cite{Galactionov-Vazquez1995}, \cite{Gilding2003, Guo 2007}, and their
references. The case $\beta \geq m$ presents special difficulties when the
free boundary $\partial \left\{ \left( t,x\right) ;u\left( t,x\right)
>0\right\} $ is a nonempty hypersurface. This set corresponds to the
so-called set of \textit{rupture points} in the study of thin films (\cite%
{Witelski2000}). This case, $\beta \geq m$, also arises in the modeling of
micro-electromechanical systems (MEMS), in which mainly $m=1$ and $\beta =2$
(\cite{Guo 2007, Pelesio2002}).

A great amount of the previous papers in the literature concern only with
the one-dimensional case. To explain some historical progresses in founding
gradient estimates for such kind of problems we start by mentioning that the
existence of weak solutions to \eqref{P} was obtained firstly by Phillips 
\cite{Phillips1987} for the case $N\geq 1$, $m=1$, and $\beta \in \left(
0,1\right) $. Later, D\'{a}vila and Montenegro \cite{Davila2005} proved an
existence result to equation \eqref{P} with $m=1$ and including also a
possible source term $f\left( u\right) $ satisfying a sublinear condition,
i.e., $f\left( u\right) \leq C\left( 1+u\right) $. They proved that the
pointwise gradient estimate: 
\begin{equation}
\left\vert {\nabla u\left( {t,x}\right) }\right\vert \leq C{u^{\frac{{%
1-\beta }}{2}}}\left( {t,x}\right) ,\mbox{ in }\left( {0,\infty }\right)
\times \Omega ,  \label{1.3}
\end{equation}%
plays a crucial role in proving the existence of solutions of \eqref{P}.
Besides, a partial uniqueness result was obtained by the same authors for a
class of solutions with initial data $u_{0}\left( x\right) \geq C\mathrm{dist%
}\left( x,\partial \Omega \right) ^{\mu }$, for $\mu \in \left( 1,2/\left(
1+\beta \right) \right) $ and some constant $C>0$ (see also \cite{Dao2016b}
for a uniqueness result in another class of solutions). The uniqueness of
solutions fails for general bounded nonnegative initial data \cite%
{Winkler2007}.

Concerning the qualitative properties satisfied by the solutions of \eqref{P}%
, one of the more peculiar facts is that the solutions may vanish after a
finite time, even starting with a positive initial data. This phenomenon
occurs by the presence of the singular absorption ${u^{-\beta }}{\chi
_{\left\{ {u>0}\right\} }}$ and can be understood as a generalization of the 
\textit{finite extinction property} which arises for not so singular
absorption terms of the form $u^{q},$ $0<q<1$. Another motivation of the
present paper is to complete the previous work \cite{Diaz2014} in which the
finite speed of propagation and other qualitative properties were proved by
means of some energy methods (see, e.g., \cite{Diaz-Veron}, \cite%
{Antontsev2002}) in the class of \textit{local} weak solutions of the more
general formulation 
\begin{equation*}
\frac{{\partial \psi \left( v\right) }}{{\partial t}}-\mathrm{div}\,\mathbf{A%
}\left( {x,t,v,Dv}\right) +B\left( {x,t,v,Dv}\right) +C\left( {x,t,v}\right)
=f\left( {x,t,v}\right) ,
\end{equation*}%
for a singular absorption term. In that paper \cite{Diaz2014} the existence
of weak solutions was merely assumed (and not proved), so our goal is to
give some answers in this complementary direction. We also point out that,
more specifically, when $m=1$, $\beta \in \left( 0,1\right) $ and we
consider equation \eqref{P} with a sublinear source term $\lambda f\left(
u\right) $, $\lambda >0$, it was shown in \cite{Montenegro2011} that there
is a real number $\lambda _{0}>0$ and a time $t_{0}>0$, such that $%
u_{\lambda }\left( t_{0},x\right) =0,\mbox{ a.e. in }\Omega ,\ \forall
\lambda \in \left( 0,\lambda _{0}\right) $: he called this phenomenon as 
\textit{the complete quenching} (see a more general statement in \cite%
{Galactionov-Vazquez1995} and \cite{Diaz2014}). Other qualitative properties
were studied in \cite{Gilding2003}.

The extension from semilinear to some one-dimensional quasilinear degenerate
equations of the $p$-Laplacian type was considered in \cite{Giacomoni} and 
\cite{Dao2016}. In that one-dimensional case, the formulation was 
\begin{equation}
\left\{ 
\begin{split}
& {\partial _{t}}u-{\partial _{x}}\left( {{{\left\vert {u_{x}}\right\vert }%
^{p-2}}{u_{x}}}\right) +{u^{-\beta }}{\chi _{\left\{ {u>0}\right\} }}=0,%
\mbox{
in }\left( {0,\infty }\right) \times \Omega , \\
& u=0,\mbox{ on }\left( {0,\infty }\right) \times \partial \Omega , \\
& u\left( {0,x}\right) ={u_{0}}\left( x\right) ,\mbox{ in }\Omega ,
\end{split}%
\right.  \label{P1}
\end{equation}%
with $p>2$, $\beta \in \left( 0,1\right) $. To obtain the existence of
solutions of \eqref{P1}, it was proved in \cite{Dao2016} the gradient
estimate: 
\begin{equation}
\left\vert {{u_{x}}\left( {t,x}\right) }\right\vert \leq C{u^{\frac{1-{\beta 
}}{p}}}\left( {t,x}\right) ,\mbox{ in }\left( {0,\infty }\right) \times
\Omega .  \label{1.6}
\end{equation}%
We note that \eqref{1.6} is a generalization of \eqref{1.3} as $p>2$.
Furthermore, it was shown in \cite{Dao2016} that any solution of equation %
\eqref{P1} must vanish after a finite time. A complete quenching result for
equation \eqref{P1} with a source $\lambda f\left( u\right) $ was obtained
by the same authors in \cite{Dao2017}. The extension of the gradient
estimates to the higher dimensional case remains today as an open problem.

As mentioned before, the first result in the literature for the
one-dimensional problem \eqref{P} with a slow diffusion ($m>1$) was due to
Kawohl and Kersner \cite{Kawohl1992a} in 1992. Once again, a suitable
gradient estimate was the key of the proof of the correct treatment of the
problem. They proved that 
\begin{equation}
\left\vert {{{\left( {{u^{\frac{{m+\beta }}{2}}}}\right) }_{x}}}\right\vert
\leq C,  \label{Estimate KK}
\end{equation}%
in the regime in which \textit{the absorption dominates the nonlinear
diffusion}, which corresponds to 
\begin{equation}
1\leq m<2+\beta .  \label{Hypo absorption dominate}
\end{equation}%
Notice that the exponent in estimate \eqref{Estimate KK} may be written also
as $1/\gamma $ with $\gamma :=2/\left( m+\beta \right) $. As a matter of
fact, in \cite{Kawohl1992a} it was also considered the opposite regime in
which \textit{the diffusion dominates over the absorption} ($m\geq 2+\beta $%
) and it was shown that the correct value for the pointwise gradient
estimate is a different value of the exponent $\gamma $ (this time $1/\left(
m-1\right) $). We will not be specially interested in such a case in this
paper but, in any case, see more details in the second part of Lemma \ref%
{lemma2}.

Our $N$-dimensional approach to derive a pointwise gradient estimate of the
type \eqref{Estimate KK} will adapt the classical Bernstein method (see,
e.g. \cite{Aronson1969, Benilan2003, Diaz1987b, Vazquez2007}) with some
ideas introduced by Ph. B\'{e}nilan (see, e.g., \cite{Aronson1979,
Benilan1981, Benilan2003}). In fact, for the special case $N=1$, we will
extend the results of \cite{Kawohl1992a} to unbounded initial data. Our
proof requires two technical additional assumptions: 
\begin{equation}
1\leq m<1+\frac{1}{\sqrt{N-1}},  \label{Hypo m and N}
\end{equation}%
and 
\begin{equation}
\beta \in \left( {{{({m-1-\sqrt{{\Delta _{m,N}}}})}_{+}},m-1+\sqrt{{\Delta
_{m,N}}}}\right) ,\mbox{ with }{\Delta _{m,N}}:=1-\left( {N-1}\right) {%
\left( {m-1}\right) ^{2}}.  \label{Hypo}
\end{equation}%
We think that such auxiliary assumptions arise merely as some limitations of
our technique of proof. The question of how to avoid them (in the framework
in which \textit{the absorption dominates the nonlinear diffusion}, $%
1<m<2+\beta $) remains an open problem for us. Nevertheless, thanks to our
technique of proof we will prove a new gradient information for the case 
\begin{equation}
\beta +m\leq 2,  \label{Hypo beta small}
\end{equation}%
(which applies to the semilinear framework) which seems to be unadvertised
in the previous literature: or the $L^{\infty }$ norm of gradient of $u^{%
\frac{m+\beta }{2}}\left( t\right) $ is smaller than ${\left\Vert {\nabla
u_{0}^{\frac{{m+\beta }}{2}}}\right\Vert _{{L^{\infty }}\left( \Omega
\right) }}$ or if the above norm is strictly smaller than this bound then it
is smaller than an universal constant $C=C\left( m,\beta ,N\right) $,
independent of $\Omega $, then it is always smaller than this constant for $%
t\in \left( 0,+\infty \right) $. Moreover, we will give some concrete
examples proving the optimality of the estimate \eqref{Estimate KK}.

For the existence of solutions we will use a monotone family of regularized
problems and we will pass to the limit thanks to the monotonicity of the
approximation of the singular nonlinear term and the contractive properties
of the semigroup associated to the (unperturbed) nonlinear diffusion over
suitable functional spaces. The pointwise gradient estimates will be
previously obtained for solutions of the regularized problems and then
extended to the solutions of \eqref{P} and \eqref{PC} by passing to the
limit in the regularizing parameters. In the case of the assumption 
\eqref{Hypo beta
small} we will pass to the limit in the gradient term $\nabla u^{m}$ by
means of a generalization of the \textit{almost everywhere gradient
convergence} technique (introduced initially for $p$-Laplace type operators
in \cite{Boccardo1992}). Finally, we will consider several qualitative
properties of solutions of \eqref{P} and \eqref{PC} implying the \textit{%
finite speed of propagation}, the \textit{uniform localization of the support%
}, and the \textit{instantaneous shrinking of the support property}. The
well known results for solutions of the porous media equation with a strong
absorption (see, e.g. \cite{Alvarez-Diaz, Diaz1987b, Kalashnikov,
Vazquez2007}) remain being valid for solutions of the problem \eqref{P}.
Here we will get some sharper estimates rather than to deal with local
solutions as in \cite{Diaz2014}. Our special interest is to analyze the
differences arising among the behavior of solutions of the porous media
equation with a strong absorption and the solutions of the porous media
equation with a singular absorption term ${u^{-\beta }}{\chi _{\left\{ {u>0}%
\right\} }}$. In the case in which the singularity is permanently excluded
of the boundary, such as for the problem \eqref{P(1)}, the behavior of the
solution (its \textquotedblleft profile\textquotedblright ) at the first
time $t=\tau _{0}$ in which there is a quenching point, was studied in \cite%
{Filippas-Guo}. In our formulation \eqref{P}, we know that there is an
permanent singularity on the boundary $\partial \Omega $ and thus our
interest is to describe the profile of the solutions near the boundary $%
\partial \Omega $. We will construct a large class of solutions showing that
their profile near the boundary follow the gradient estimate proved in this
paper. So, such gradient estimates are sharp. Some commentaries on the case $%
\beta \geq m$ will be also given at the end of the paper.

\subsection{Main results}

Let us first introduce the notion of weak solution that we use for the case
of $\Omega $ bounded and bounded initial data.

\begin{definition}
\label{definition1} Let $u_{0}\in L^{\infty }\left( \Omega \right) $, $%
u_{0}\geq 0$. A nonnegative function $u\left( t,x\right) $ is called a \emph{%
weak solution} of \eqref{P} if $u\in \mathcal{C}\left( {\left[ {0,\infty }%
\right) ;{L^{1}}\left( \Omega \right) }\right) \cap {L^{\infty }}\left( {%
\left( {0,\infty }\right) \times \Omega }\right) $, ${u^{-\beta }}{\chi
_{\left\{ {u>0}\right\} }}\in {L^{1}}\left( {\left( {0,T}\right) \times
\Omega }\right) ,$ $u^{m}\in L^{2}\left( 0,T;H_{0}^{1}(\Omega \right) )$ for
any $T>0$, and $u$ satisfy \eqref{P} in the sense of distributions $\mathcal{%
D}^{\prime }\left( \left( 0,\infty \right) \times \Omega \right) $, i.e., 
\begin{equation*}
\int_{0}^{\infty }{\int_{\Omega }{\left( {\ -u{\varphi _{t}}+\nabla {u^{m}}%
\cdot \nabla \varphi +{u^{-\beta }}{\chi _{\left\{ {u>0}\right\} }}\varphi }%
\right) dxdt}}=0,\ \forall \varphi \in \mathcal{C}_{c}^{\infty }\left( {%
\left( {0,\infty }\right) \times \Omega }\right) .
\end{equation*}
\end{definition}

Any weak solution is also a \textit{very weak solution} to equation \eqref{P}
(see e.g., \cite{Aronson1982, Kawohl1992a, Vazquez2007}). Since the reaction
term ${u^{-\beta }}{\chi _{\left\{ {u>0}\right\} }}$ is required to be in $%
L^{1}\left( \left( 0,\infty \right) \times \Omega \right) ,$ a natural
weaker notion of solution will be used sometimes in the paper for the class
of nonnegative initial data which are merely in $L^{1}\left( \Omega \right) $%
:

\begin{definition}
\label{definition2} Let $u_{0}\in L^{1}\left( \Omega \right) $, $u_{0}\geq 0$%
, and $T>0$. A nonnegative function $u\in \mathcal{C}\left( \left[ 0,T\right]
;L^{1}\left( \Omega \right) \right) $ is called a \emph{$L^{1}$-mild solution%
} of \eqref{P} if ${u^{-\beta }}{\chi _{\left\{ {u>0}\right\} }}\in {L^{1}}%
\left( {\left( {0,T}\right) \times \Omega }\right) $ and $u$ coincides with
the unique $L^{1}$-mild solution of the problem 
\begin{equation}
\left\{ 
\begin{split}
& {\partial _{t}}u-\Delta {u^{m}}=f,\mbox{ in }\left( {0,T}\right) \times
\Omega , \\
& u=0,\mbox{ on }\left( {0,T}\right) \times \partial \Omega , \\
& u\left( {0,x}\right) ={u_{0}}\left( x\right) ,\mbox{ in }\Omega ,
\end{split}%
\right.  \label{eq nonhomogeneous}
\end{equation}%
where $f\left( {t,x}\right) :=-{u^{-\beta }}\left( {t,x}\right) {\chi
_{\left\{ {u>0}\right\} }}\left( {t,x}\right) $ on $\left( 0,T\right) \times
\Omega $.
\end{definition}

As a matter of fact, a weaker notion of solutions can be obtained when
introducing \textit{the distance to the boundary} as a weight: ${u_{0}}\in
L_{\delta }^{1}\left( \Omega \right) =\left\{ {v\in L_{\mathrm{loc}%
}^{1}\left( \Omega \right) ;\int_{\Omega }{v\left( x\right) \delta \left(
x\right) dx}<\infty }\right\} $, where 
\begin{equation*}
\delta \left( x\right) =d\left( x,\partial \Omega \right) .
\end{equation*}

\begin{definition}
\label{definition3} Let $u_{0}\in L_{\delta }^{1}\left( \Omega \right) $, $%
u_{0}\geq 0$, and $T>0$. A nonnegative function $u\in \mathcal{C}\left( %
\left[ 0,T\right] ;L_{\delta }^{1}\left( \Omega \right) \right) $ is called
a \emph{$L_{\delta }^{1}$-mild solution} of \eqref{P} if ${u^{-\beta }}{\chi
_{\left\{ {u>0}\right\} }}\in {L^{1}}\left( {0,T;L_{\delta }^{1}\left(
\Omega \right) }\right) $ and $u$ coincides with the unique $L_{\delta }^{1}$%
-mild solution of the problem \eqref{eq nonhomogeneous}, with $f:=-{%
u^{-\beta }}{\chi _{\left\{ {u>0}\right\} }}$.
\end{definition}

We recall that the notion of mild solution of the problem for the
non-homogeneous problem \eqref{eq nonhomogeneous} is well-defined thanks to
the fact that the nonlinear diffusion operator $-\Delta u^{m}$ (with
Dirichlet boundary conditions) is a $m$-accretive operator in $L^{1}\left(
\Omega \right) $ with a dense domain (see, e.g., \cite{Benilan1981,
Benilan1994, Vazquez2007} and their references). The similar properties of
this operator on the space $L_{\delta }^{1}\left( \Omega \right) $ will be
shown in this paper as easy consequences of well-known results (\cite{Brezis
1971, Brezis 1996, Diaz2010, Rakotoson2011, Veron2004} and Section 6.6 of 
\cite{Vazquez2007}). In fact, there are other equivalent formulations for
very weak solutions obtained as $L_{\delta }^{1}$-mild solution of the
problem \eqref{eq
nonhomogeneous}. One formulation which is specially useful for our purposes
starts by introducing the auxiliary equivalent weight function $\zeta \left(
x\right) $, $\zeta \in C^{\infty }\left( \Omega \right) \cap C^{1}\left( 
\overline{\Omega }\right) $, $\zeta >0$, given as the unique solution of the
problem 
\begin{equation}
\left\{ 
\begin{split}
-\Delta {\zeta }& =1,\text{ in }\Omega , \\
{\zeta }& =0,\text{ on }\partial \Omega .
\end{split}%
\right.  \label{Eq Function xi}
\end{equation}%
It is well known that 
\begin{equation}
\underline{C}\delta \left( x\right) \leq \zeta \left( x\right) \leq 
\overline{C}\delta \left( x\right) ,\mbox{ for any }x\in \Omega ,
\label{(1.14)}
\end{equation}%
for some positive constants $\underline{C}<\overline{C}$, so that $L_{\delta
}^{1}\left( \Omega \right) =L_{\zeta }^{1}\left( \Omega \right) $. Then, it
is easy to see that every \textit{$L_{\delta }^{1}$-mild solution} of %
\eqref{P} is a very weak solution of the problem \eqref{eq nonhomogeneous}
in the sense that $u\in \mathcal{C}\left( \left[ 0,T\right] ;L_{\delta
}^{1}\left( \Omega \right) \right) $, $u\geq 0$, $u^{m}\in L^{1}\left(
\left( 0,T\right) \times \Omega \right) $, $f=-{u^{-\beta }}{\chi _{\left\{ {%
u>0}\right\} }}\in {L^{1}}\left( {0,T;L_{\delta }^{1}\left( \Omega \right) }%
\right) $, and for any $t\in \left[ 0,T\right] $, 
\begin{equation*}
\int_{\Omega }{u\left( {t,x}\right) \zeta \left( x\right) dx}+\int_{0}^{t}{%
\int_{\Omega }{{u^{m}}\left( {t,x}\right) dxdt}}=\int_{\Omega }{{u_{0}}%
\left( x\right) \zeta \left( x\right) dx}+\int_{0}^{t}{\int_{\Omega }{%
f\left( {t,x}\right) \zeta \left( x\right) dxdt}}.
\end{equation*}%
In what follows, our main interest will deal with the cases of $N\geq 2$,
and $m>1$ since the two other cases ($N=1$, $m\geq 1$; and $N\geq 1$, $m=1$)
were studied in \cite{Kawohl1992a} and \cite{Phillips1987}, respectively. We
also mention that some singular reaction terms were considered previously in
the literature for the case of $m\in \left( 0,1\right) $ (see, e.g., \cite%
{Dai2006, Deng1992}). Some of our results also hold for $m\in \left(
0,1\right) $ but we will not pursuit such a goal in this paper.

Our main result in this paper is the following one:

\begin{theorem}
\label{theorem1}

\begin{itemize}
\item[i)] Let $u_{0}\in L_{\delta }^{1}\left( \Omega \right) $, $u_{0}\geq 0$%
. Assume $m\geq 1$ and $\beta \in \left( 0,m\right) $. Then, problem %
\eqref{P} has a \emph{maximal} $L_{\delta }^{1}$-mild solution $u$. Moreover
if $u_{0}\in L^{1}\left( \Omega \right) $ then $u$ is also the \emph{maximal}
$L^{1}$-mild solution.

\item[ii)] Let $u_{0}\in L_{\delta }^{1}\left( \Omega \right) $, $u_{0}\geq
0 $ and assume \eqref{Hypo absorption dominate}, \eqref{Hypo m and N}, and %
\eqref{Hypo}. Then 
\begin{equation*}
{\left\Vert {\nabla {u^{\frac{{m+\beta }}{2}}}\left( t\right) }\right\Vert _{%
{L^{\infty }}\left( \Omega \right) }}\leq C\left( {\frac{1}{{t^{\omega }}}+1}%
\right) ,\mbox{ a.e. }t\in \left( {0,+\infty }\right) ,
\end{equation*}%
for some positive constants $\omega =\omega \left( m,\beta ,N\right) $ and $%
C=C\left( m,\beta ,N,\Omega \right) $ if $m>1$, $C=C\left( m,\beta
,N,\left\Vert u_{0}\right\Vert _{L_{\delta }^{1}\left( \Omega \right)
}\right) $ if $m=1$. Moreover the \emph{maximal} $L^{1}$-mild solution is H%
\"{o}lder continuous on $({0,T]}\times \overline{\Omega }$.

\item[iii)] Let $u_{0}\in L_{\delta }^{1}\left( \Omega \right) $, $u_{0}\geq
0$ such that $\nabla u_{0}^{\frac{m+\beta }{2}}\in L^{\infty }\left( \Omega
\right) $ and assume $m\geq 1$, \eqref{Hypo absorption dominate}, %
\eqref{Hypo m and N}, \eqref{Hypo} and \eqref{Hypo beta small}. Then 
\begin{equation*}
{\left\Vert {\nabla {u^{\frac{{m+\beta }}{2}}}\left( t\right) }\right\Vert _{%
{L^{\infty }}\left( \Omega \right) }}\leq \max \left\{ {{{\left\Vert {\nabla
u_{0}^{\frac{{m+\beta }}{2}}}\right\Vert }_{{L^{\infty }}\left( \Omega
\right) }},\frac{{\left( {m+\beta }\right) \sqrt{2+\beta -m}}}{\sqrt{%
2m\left( {{\Delta _{m,N}}-{{\left( {\beta +1-m}\right) }^{2}}}\right) }}}%
\right\} ,
\end{equation*}
\end{itemize}

$\mbox{ a.e. }t\in \left( {0,+\infty }\right) .$
\end{theorem}

\bigskip

We point out that in the rest of the paper we will denote by $C$ different
positive constants, possibly changing from line to line. Furthermore, any
constant, depending on some parameters will be emphasized by a parentheses
indicating such a dependence: for instance, $C=C\left( m,\beta ,N\right) $
will mean that $C$ depends only on $m$, $\beta $, $N$.

\bigskip

\begin{remark}
Concerning the one-dimensional quasilinear case, $m>1$, Theorem \ref%
{theorem1} extends the results by Kawohl and Kersner \cite{Kawohl1992a} to a
class of more general initial data. Notice also that the gradient estimate
given by in part iii) is new with respect to the paper \cite{Kawohl1992a}
and also with respect to the literature on the semilinear problem. It can be
useful for many different purposes (for instance to control possible
approximating algorithms when there are some additional perturbations in the
right hand side of the equation, and so on).
\end{remark}

\begin{remark}
We emphasize that the gradient estimates prove (see Proposition \ref%
{proposition1} below) that in fact $u^{\frac{m+1}{2}}$ is H\"{o}lder
continuous on $\left( 0,\infty \right) \times \overline{\Omega }$ (and in
fact also on $[0,\infty )\times \overline{\Omega }$ provided that $u_{0}^{%
\frac{m+1}{2}}$ is also H\"{o}lder continuous on $\overline{\Omega }$ and $%
\nabla u_{0}^{\frac{m+\beta }{2}}\in L^{\infty }\left( \Omega \right) $).
\end{remark}

The existence of solutions to the Cauchy problem \eqref{PC} can be obtained
as a consequence of Theorem \ref{theorem1}. Moreover, the above gradient
estimates hold on $L^{\infty }\left( \mathbb{R}^{N}\right) $ for a.e. $t\in
\left( 0,T\right) $ (see Theorem \ref{theorem2 copy(1)} below).

This paper is organized as follows. In the next section, we will prove the
pointwise gradient estimates of solutions of a regularized version of
equation \eqref{P}. Section 3 is devoted to prove Theorem 1 and its
application to the study of the Cauchy problem \eqref{PC}. Different
qualitative properties will be considered in the final Section 4.

\section{Technical lemmas}

In this section, we will adapt to our framework the classical Bernstein's
technique and some ideas of Ph. B\'{e}nilan and his collaborators, in order
to obtain a gradient estimate of the type $\left\vert \nabla u^{1/\gamma
}\right\vert \leq C$ with $\gamma :=2/\left( m+\beta \right) $. Let $\psi
\in \mathcal{C}^{\infty }\left( \mathbb{R}:\left[ 0,1\right] \right) $ be a
non-decreasing real function such that 
\begin{equation*}
\psi \left( s\right) =\left\{ 
\begin{split}
& 0,\mbox{ if }s\leq 1, \\
& 1,\mbox{ if }s\geq 2.
\end{split}%
\right.
\end{equation*}%
For every $\varepsilon >0$, we define $g_{\varepsilon }\left( s\right)
:=s^{-\beta }\psi _{\varepsilon }\left( s\right) $, where $\psi
_{\varepsilon }\left( s\right) =\psi \left( s/\varepsilon \right) $, for $%
s\in \mathbb{R}$. It is straightforward to check that $g_{\varepsilon }$ is
a globally Lipschitz function for any $\varepsilon >0$.

Now, for every $\varepsilon >0$ and $\eta >0$, we consider the regularized
version of problem \eqref{P} given by

\begin{equation*}
(P_{\varepsilon ,\eta })=\left\{ 
\begin{split}
& {\partial _{t}}u-\Delta {u^{m}}+{g_{\varepsilon }}\left( u\right) =0,%
\mbox{ in }\left( {0,\infty }\right) \times \Omega , \\
& u=\eta ,\mbox{ on }\left( {0,\infty }\right) \times \partial \Omega , \\
& u\left( {0,x}\right) ={u_{0}}\left( x\right) +\eta ,\mbox{ in }\Omega .
\end{split}%
\right.
\end{equation*}

The main goal of this section is to get some pointwise estimates for $\nabla
u_{\varepsilon ,\eta }$ (with $u_{\varepsilon ,\eta }$ the unique solution
of $(P_{\varepsilon ,\eta })$) which will allow to pass to the limit, as $%
\eta ,\,\varepsilon \downarrow 0$, to prove the gradient estimates indicated
in Theorem \ref{theorem1}.

We start by showing a general auxiliary result which is useful to handle
expressions containing terms of the type $\left\vert \nabla u\right\vert
^{2}\Delta u$ arising in the study of gradient estimates in the
multi-dimensional case. Our proof corresponds to a slight generalization of B%
\'{e}nilan's ideas (see, e.g., \cite{Aronson1979, Benilan1981} and the
application made in \cite{Benachour2016}).

\begin{lemma}
\label{lemma1} Let $u\in \mathcal{C}^{2}(\mathbb{R}^{N},\mathbb{R})$, and $%
g\in \mathcal{C}^{1}\left( \mathbb{R},\left[ 0,\infty \right) \right) $.
Then, the following inequality holds over the set $\left\{ {x\in {\mathbb{R}%
^{N}};g\left( {u\left( x\right) }\right) \neq 0}\right\} $: 
\begin{equation*}
g\left( u\right) {\left\vert {{D^{2}}u}\right\vert ^{2}}+g^{\prime }\left(
u\right) \left( {\frac{1}{2}\nabla u\cdot \nabla \left( {{{\left\vert {%
\nabla u}\right\vert }^{2}}}\right) -{{\left\vert {\nabla u}\right\vert }^{2}%
}\Delta u}\right) \geq -\frac{{\left( {N-1}\right) {{{g^{\prime }\left(
u\right) }}^{2}{\left\vert {\nabla u}\right\vert }^{4}}}}{{4g\left( u\right) 
}}.
\end{equation*}
\end{lemma}

\begin{proof}[Proof of Lemma \ref{lemma1}]
Set $w:= \left|\nabla u\right|^2$ and denote by $\mathcal{S}\left(g,u\right)$ the left-hand side of the wanted inequality. Then $\mathcal{S}\left(g,u\right)$ can be rewritten as
\begin{align*}
\mathcal{S} \left( {g,u} \right) = g\left( u \right){\left| {{D^2}u} \right|^2} + g'\left( u \right)\left( {\frac{1}{2}\nabla u \cdot \nabla w - w\Delta u} \right).
\end{align*}
As in \cite{Benachour2016}, we can adapt the B\'enilan's method presented in \cite{Benilan1981} in the following way:
\begin{align*}
\mathcal{S} \left( {g,u} \right) =&\ g\left( u \right)\sum\limits_{i,j = 1}^N {{{\left( {{\partial _{ij}}u} \right)}^2}}  + g'\left( u \right)\left( {\sum\limits_{i,j = 1}^N {{\partial _i}u{\partial _j}u{\partial _{ij}}u}  - w\sum\limits_{i = 1}^N {\partial _i^2u} } \right)\\
 =&\ g\left( u \right)\sum\limits_{i = 1}^N {\left[ {{{\left( {\partial _i^2u} \right)}^2} + \frac{{g'}}{g}\left( u \right)\left( {{{\left( {{\partial _i}u} \right)}^2} - w} \right)\partial _i^2u} \right]}  + g\left( u \right)\sum\limits_{i \ne j} {\left[ {{{\left( {{\partial _{ij}}u} \right)}^2} + \frac{{g'}}{g}\left( u \right){\partial _i}u{\partial _j}u{\partial _{ij}}u} \right]} \\
 =&\ g\left( u \right)\sum\limits_{i = 1}^N {{{\left[ {\partial _i^2u + \frac{{g'}}{{2g}}\left( u \right)\left( {{{\left( {{\partial _i}u} \right)}^2} - w} \right)} \right]}^2}}  - \frac{{g\left( u \right)}}{4}\sum\limits_{i = 1}^N {{{\left( {\frac{{g'}}{g}} \right)}^2}\left( u \right){{\left( {{{\left( {{\partial _i}u} \right)}^2} - w} \right)}^2}} \\
& + g\left( u \right)\sum\limits_{i \ne j} {{{\left( {{\partial _{ij}}u + \frac{{g'}}{{2g}}\left( u \right){\partial _i}u{\partial _j}u} \right)}^2}}  - \frac{{g\left( u \right)}}{4}\sum\limits_{i \ne j} {{{\left( {\frac{{g'}}{g}} \right)}^2}\left( u \right){{\left( {{\partial _i}u} \right)}^2}{{\left( {{\partial _j}u} \right)}^2}} \\
 \ge & - \frac{{{{\left( {g'} \right)}^2}}}{{4g}}\left( u \right)\left[ {\sum\limits_{i = 1}^N {{{\left( {{{\left( {{\partial _i}u} \right)}^2} - w} \right)}^2}}  + \sum\limits_{i \ne j} {{{\left( {{\partial _i}u} \right)}^2}{{\left( {{\partial _j}u} \right)}^2}} } \right] = - \frac{{\left( {N - 1} \right){{\left( {g'} \right)}^2}}}{{4g}}\left( u \right){w^2} ,
\end{align*}
which completes the proof.
\end{proof}

\bigskip

Given $u_{0}\in \mathcal{C}_{c}^{1}\left( \Omega \right) $, $u_{0}\geq 0$, $%
u_{0}\neq 0$, $m\geq 1$ and $0<\eta \leq \min \left\{ {\varepsilon ,{{%
\left\Vert {u_{0}}\right\Vert }_{\infty }}}\right\} $, the existence and
uniqueness of a classical solution $u_{\varepsilon ,\eta }$ of $%
(P_{\varepsilon ,\eta })$ is a well-known result (see, e.g., \cite%
{Ladyvzenskaja1968}). Moreover, the comparison principle applies and thus 
\begin{equation*}
\eta \leq {u_{\varepsilon ,\eta }}\left( {t,x}\right) \leq {\left\Vert {{%
u_{0}}}\right\Vert _{\infty }}+\eta \leq 2{\left\Vert {{u_{0}}}\right\Vert
_{\infty }},\mbox{ in }\left( {0,\infty }\right) \times \Omega .
\end{equation*}%
We will prove the gradient estimates in a separate way: first for the case $%
N\geq 2$ and then for $N=1$.

\begin{lemma}
\label{lemma2} Let $u_{0}\in \mathcal{C}_{c}^{1}\left( \Omega \right) $ be
nonnegative, $0<\eta \leq \min \left\{ {\varepsilon ,{{\left\Vert {u_{0}}%
\right\Vert }_{\infty }}}\right\} .$ Let $N\geq 2$ and $m\geq 1$ be such
that $\Delta _{m,N}>0$. Define $\gamma :=\frac{2}{m+\beta }$ and assume %
\eqref{Hypo}. Then there is a positive constant $C=C\left( m,\beta ,N\right)$
such that 
\begin{equation}
{\left\vert {\nabla u_{\varepsilon ,\eta }^{1/\gamma }\left( {t,x}\right) }
\right\vert ^{2}}\leq C\left( {{t^{-1}}\left\Vert {u_{0}}\right\Vert _{{%
L^{\infty }}\left( \Omega \right) }^{1+\beta }+1}\right) ,\mbox{ in }\left( {%
0,\infty }\right) \times \Omega .  \label{GE}
\end{equation}%
In addition, if one assumes \eqref{Hypo beta small} and $\nabla
u_0^{1/\gamma} \in L^\infty\left(\Omega\right)$, then 
\begin{equation}  \label{GE1}
\left| {\nabla u_{\varepsilon ,\eta }^{1/\gamma }\left( {t,x} \right)}
\right| \le \max \left\{ {{{\left\| {\nabla u_0^{\frac{{m + \beta }}{2}}}
\right\|}_{{L^\infty }\left( \Omega \right)}},\frac{{\left( {m + \beta }
\right)\sqrt {2 + \beta - m} }}{{\sqrt {2m\left( {{\Delta _{m,N}} - {{\left( 
{\beta + 1 - m} \right)}^2}} \right)} }}} \right\}, \mbox{ in } \left( {%
0,\infty } \right) \times \Omega .
\end{equation}
\end{lemma}

\begin{proof}
Let $h_{\varepsilon ,\eta }:=u_{\varepsilon ,\eta
}^{1/\gamma }$. Then, $h_{\varepsilon ,\eta }$ satisfies the following
equation:
\begin{equation}
{\partial _{t}}h_{\varepsilon ,\eta }-mh_{\varepsilon ,\eta }^{\gamma \left(
{m-1}\right) }\Delta {h_{\varepsilon ,\eta }}-m\left( {m\gamma -1}\right)
h_{\varepsilon ,\eta }^{\gamma \left( {m-1}\right) -1}{\left\vert {\nabla {%
h_{\varepsilon ,\eta }}}\right\vert ^{2}}+{\gamma ^{-1}}{\psi _{\varepsilon }%
}\left( {h_{\varepsilon ,\eta }^{\gamma }}\right) h_{\varepsilon ,\eta
}^{1-\gamma \left( {1+\beta }\right) }=0.  \label{h}
\end{equation}%
Differentiating in \eqref{h} with respect to the variable $x$, we obtain
\begin{align}
{\partial _{t}}\nabla {h_{\varepsilon ,\eta }}& -mh_{\varepsilon ,\eta
}^{\gamma \left( {m-1}\right) }\nabla \Delta {h_{\varepsilon ,\eta }}%
=m\gamma \left( {m-1}\right) h_{\varepsilon ,\eta }^{\gamma \left( {m-1}%
\right) -1}\Delta {h_{\varepsilon ,\eta }}\nabla {h_{\varepsilon ,\eta }}
\notag \\
& +m\left( {m\gamma -1}\right) \left( {\gamma \left( {m-1}\right) -1}\right)
h_{\varepsilon ,\eta }^{\gamma \left( {m-1}\right) -2}{\left\vert {\nabla {%
h_{\varepsilon ,\eta }}}\right\vert ^{2}}\nabla {h_{\varepsilon ,\eta }}
\notag \\
& +m\left( {m\gamma -1}\right) h_{\varepsilon ,\eta }^{\gamma \left( {m-1}%
\right) -1}\nabla \left( {{{\left\vert {\nabla {h_{\varepsilon ,\eta }}}%
\right\vert }^{2}}}\right) -{\psi _{\varepsilon }}^{\prime }\left( {%
h_{\varepsilon ,\eta }^{\gamma }}\right) h_{\varepsilon ,\eta }^{-\beta
\gamma }\nabla {h_{\varepsilon ,\eta }}  \notag \\
& -{\gamma ^{-1}}\left( {1-\gamma \left( {1+\beta }\right) }\right) {\psi
_{\varepsilon }}\left( {h_{\varepsilon ,\eta }^{\gamma }}\right)
h_{\varepsilon ,\eta }^{-\gamma \left( {1+\beta }\right) }\nabla {%
h_{\varepsilon ,\eta }},\mbox{ in }\left( {0,\infty }\right) \times \Omega .
\label{2.6}
\end{align}%
For any $0<\tau <T<\infty $, let $\zeta \in C^{\infty }\left( \mathbb{R}:%
\left[ 0,1\right] \right) $ be a cut-off function such that
\begin{equation*}
\zeta \left( t\right) =\left\{
\begin{split}
& 1,\mbox{ if }t\in \left[ {\tau ,T}\right] , \\
& 0,\mbox{ if }t\notin \left( \frac{\tau }{2},T+\frac{\tau }{2}\right) ,
\end{split}%
\right. \mbox{ and }\left\vert {\zeta '}\right\vert \leq \frac{{{c_{0}}%
}}{\tau }\mbox{ for some positive constant }{c_{0}}.
\end{equation*}%
Consider now the function ${v_{\varepsilon ,\eta }}\left( {t,x}\right)
:=\zeta \left( t\right) {\left\vert {\nabla {h_{\varepsilon ,\eta }}\left( {%
t,x}\right) }\right\vert ^{2}}$. Let $M:=\max_{\left[ 0,\infty \right)
\times \overline{\Omega }}v_{\varepsilon ,\eta }$. It is enough to assume $%
M>0$, otherwise it is clear that $\nabla h_{\varepsilon ,\eta }\equiv 0$,
likewise $\nabla u_{\varepsilon ,\eta }\equiv 0$. Therefore, there is a
point $\left( t_{0},x_{0}\right) \in \left( \tau /2,T+\tau /2\right) \times
\Omega $ such that $v_{\varepsilon ,\eta }\left( t_{0},x_{0}\right) =M$
(since $v_{\varepsilon ,\eta }=0$ on $\left[ 0,\infty \right) \times
\partial \Omega $). As a consequence, one has
\begin{equation}
\nabla \left( {{{\left\vert {\nabla {h_{\varepsilon ,\eta }}}\right\vert }%
^{2}}}\right) =0\mbox{ and }{\partial _{t}}{v_{\varepsilon ,\eta }}%
-mh_{\varepsilon ,\eta }^{\gamma \left( {m-1}\right) }\Delta {v_{\varepsilon
,\eta }}\geq 0,\mbox{ at }\left( {{t_{0}},{x_{0}}}\right) .  \label{2.7}
\end{equation}%
This implies
\begin{equation*}
\zeta ^{\prime }{\left\vert {\nabla {h_{\varepsilon ,\eta }}}\right\vert ^{2}%
}+2\zeta \nabla {h_{\varepsilon ,\eta }}\cdot {\partial _{t}}\nabla {%
h_{\varepsilon ,\eta }}\geq 2m\zeta h_{\varepsilon ,\eta }^{\gamma \left( {%
m-1}\right) }\left( {{{\left\vert {{D^{2}}{h_{\varepsilon ,\eta }}}%
\right\vert }^{2}}+\nabla {h_{\varepsilon ,\eta }}\cdot \nabla \Delta {%
h_{\varepsilon ,\eta }}}\right) ,\mbox{ at }\left( {{t_{0}},{x_{0}}}\right) ,
\end{equation*}%
or, equivalently,
\begin{equation*}
\zeta \nabla {h_{\varepsilon ,\eta }}\cdot \left( {{\partial _{t}}\nabla {%
h_{\varepsilon ,\eta }}-mh_{\varepsilon ,\eta }^{\gamma \left( {m-1}\right)
}\nabla \Delta {h_{\varepsilon ,\eta }}}\right) \geq -\frac{{\zeta ^{\prime }%
}}{2}{\left\vert {\nabla {h_{\varepsilon ,\eta }}}\right\vert ^{2}}+m\zeta
h_{\varepsilon ,\eta }^{\gamma \left( {m-1}\right) }{\left\vert {{D^{2}}{%
h_{\varepsilon ,\eta }}}\right\vert ^{2}},\mbox{ at }\left( {{t_{0}},{x_{0}}}%
\right) .
\end{equation*}%
Combining this with \eqref{2.6} and the former version of \eqref{2.7}, we
obtain
\begin{align}
& m\left( {m\gamma -1}\right) \left( {1-\gamma \left( {m-1}\right) }\right)
\zeta h_{\varepsilon ,\eta }^{\gamma \left( {m-1}\right) -2}{\left\vert {%
\nabla {h_{\varepsilon ,\eta }}}\right\vert ^{4}}  \notag \\
\hspace{0.5cm}\leq & \ \frac{{\zeta ^{\prime }}}{2}{\left\vert {\nabla {%
h_{\varepsilon ,\eta }}}\right\vert ^{2}}+m\gamma \left( {m-1}\right) \zeta
h_{\varepsilon ,\eta }^{\gamma \left( {m-1}\right) -1}\Delta {h_{\varepsilon
,\eta }}{\left\vert {\nabla {h_{\varepsilon ,\eta }}}\right\vert ^{2}}%
-m\zeta h_{\varepsilon ,\eta }^{\gamma \left( {m-1}\right) }{\left\vert {{%
D^{2}}{h_{\varepsilon ,\eta }}}\right\vert ^{2}}  \notag \\
\hspace{0.5cm}& -\zeta {\psi _{\varepsilon }}^{\prime }\left( {%
h_{\varepsilon ,\eta }^{\gamma }}\right) h_{\varepsilon ,\eta }^{-\beta
\gamma }{\left\vert {\nabla {h_{\varepsilon ,\eta }}}\right\vert ^{2}}%
+\left( {1+\beta -{\gamma ^{-1}}}\right) \zeta {\psi _{\varepsilon }}\left( {%
h_{\varepsilon ,\eta }^{\gamma }}\right) h_{\varepsilon ,\eta }^{-\gamma
\left( {1+\beta }\right) }{\left\vert {\nabla {h_{\varepsilon ,\eta }}}%
\right\vert ^{2}},\mbox{ at }\left( {{t_{0}},{x_{0}}}\right) .  \label{2.8}
\end{align}%
From \eqref{2.7}, applying Lemma \ref{lemma1} to $g\left( s\right)
=s^{\gamma \left( m-1\right) }$ we get
\begin{align*}
& h_{\varepsilon ,\eta }^{\gamma \left( {m-1}\right) }{\left\vert {{D^{2}}{%
h_{\varepsilon ,\eta }}}\right\vert ^{2}}-\gamma \left( {m-1}\right)
h_{\varepsilon ,\eta }^{\gamma \left( {m-1}\right) -1}{\left\vert {\nabla {%
h_{\varepsilon ,\eta }}}\right\vert ^{2}}\Delta {h_{\varepsilon ,\eta }} \\
& =h_{\varepsilon ,\eta }^{\gamma \left( {m-1}\right) }{\left\vert {{D^{2}}{%
h_{\varepsilon ,\eta }}}\right\vert ^{2}}+\gamma \left( {m-1}\right)
h_{\varepsilon ,\eta }^{\gamma \left( {m-1}\right) -1}\left( {\frac{1}{2}%
\nabla {h_{\varepsilon ,\eta }}\cdot \nabla \left( {{{\left\vert {\nabla {%
h_{\varepsilon ,\eta }}}\right\vert }^{2}}}\right) -{{\left\vert {\nabla {%
h_{\varepsilon ,\eta }}}\right\vert }^{2}}\Delta {h_{\varepsilon ,\eta }}}%
\right) \\
& \geq -\frac{1}{4}{\gamma ^{2}}\left( {N-1}\right) {\left( {m-1}\right) ^{2}%
}h_{\varepsilon ,\eta }^{\gamma \left( {m-1}\right) -2}{\left\vert {\nabla {%
h_{\varepsilon ,\eta }}}\right\vert ^{4}},\mbox{ at }\left( {{t_{0}},{x_{0}}}%
\right) .
\end{align*}%
A combination of this equality, \eqref{2.8}, and $\nabla h_{\varepsilon
,\eta }\left( t_{0},x_{0}\right) \neq 0$ implies
\begin{align}
& m\left[ {\left( {m\gamma -1}\right) \left( {1-\gamma \left( {m-1}\right) }%
\right) -{\gamma ^{2}}\left( {N-1}\right) {{\left( {m-1}\right) }^{2}}/4}%
\right] \zeta h_{\varepsilon ,\eta }^{\gamma \left( {m-1}\right) -2}{%
\left\vert {\nabla {h_{\varepsilon ,\eta }}}\right\vert ^{2}}  \notag \\
& \hspace{0.5cm}\leq \frac{{\zeta ^{\prime }}}{2}-\zeta {\psi _{\varepsilon }%
}^{\prime }\left( {h_{\varepsilon ,\eta }^{\gamma }}\right) h_{\varepsilon
,\eta }^{-\beta \gamma }+\left( {1+\beta -{\gamma ^{-1}}}\right) \zeta {\psi
_{\varepsilon }}\left( {h_{\varepsilon ,\eta }^{\gamma }}\right)
h_{\varepsilon ,\eta }^{-\gamma \left( {1+\beta }\right) },\mbox{ at }\left(
{{t_{0}},{x_{0}}}\right) .  \label{2.9}
\end{align}%
Denote
\begin{equation*}
\mathcal{B}:=m\left[ {\left( {m\gamma -1}\right) \left( {1-\gamma \left( {m-1%
}\right) }\right) -\frac{1}{4}{\gamma ^{2}}\left( {N-1}\right) {{\left( {m-1}%
\right) }^{2}}}\right] =\frac{{m\left[ {{\Delta _{m,N}}-{{\left( {\beta +1-m}%
\right) }^{2}}}\right] }}{{{{\left( {m+\beta }\right) }^{2}}}}.
\end{equation*}%
Note that the assumption \eqref{Hypo} on $\beta $ implies that $\mathcal{B}%
>0$. Since $\psi _{\varepsilon }^{\prime }\geq 0$, it is clear that the
second term on the right hand side of \eqref{2.9} is non-positive. As a
consequence, we get
\begin{equation*}
\mathcal{B}{v_{\varepsilon ,\eta }}=\mathcal{B}\zeta {\left\vert {\nabla {%
h_{\varepsilon ,\eta }}}\right\vert ^{2}}\leq \frac{{\zeta ^{\prime }}}{2}%
h_{\varepsilon ,\eta }^{2-\gamma \left( {m-1}\right) }+\left( {1+\beta -{%
\gamma ^{-1}}}\right) \zeta {\psi _{\varepsilon }}\left( {h_{\varepsilon
,\eta }^{\gamma }}\right) h_{\varepsilon ,\eta }^{2-\gamma \left( {m+\beta }%
\right) },\mbox{ at }\left( {{t_{0}},{x_{0}}}\right) .
\end{equation*}%
Note that $2-\gamma \left( m-1\right) =2\left( 1+\beta \right) /\left(
m+\beta \right) >0$ and $1+\beta -\gamma ^{-1}=\left( 2+\beta -m\right) /2>0$
(since $\Delta _{m,N}>0$ implies $m<1+1/\sqrt{N-1}\leq 2$ for all $N\geq 2$%
), the last inequality then implies
\begin{equation*}
M\leq \frac{1}{{2\mathcal{B}}}\left[ {\frac{{{c_{0}}}}{\tau }{{\left( {2{{%
\left\Vert {{u_{0}}}\right\Vert }_{\infty }}}\right) }^{1+\beta }}+2+\beta -m%
}\right] .
\end{equation*}%
Since $v_{\varepsilon ,\eta }\left( t,x\right) \leq M$ in $\left( 0,\infty
\right) \times \Omega $, the last inequality implies, in particular, at $%
t=\tau $:
\begin{equation*}
{\left\vert {\nabla u_{\varepsilon ,\eta }^{1/\gamma }\left( {\tau ,x}%
\right) }\right\vert ^{2}}\leq \frac{1}{{2\mathcal{B}}}\left( {{2^{1+\beta }}%
{c_{0}}{\tau ^{-1}}\left\Vert {{u_{0}}}\right\Vert _{\infty }^{1+\beta
}+2+\beta -m}\right) ,\ \forall x\in \Omega .
\end{equation*}%
The proof of the second statement is a small variation of the above case.
For any $\tau >0$, it suffices to make a slight modification by replacing
the cut-off function $\zeta \left( t\right) $ by $\overline{\zeta }\left(
t\right) \in C^{\infty }\left( \mathbb{R}:\left[ 0,1\right] \right) $
defined by
\begin{equation*}
\bar{\zeta}\left( t\right) =\left\{
\begin{split}
& 1,\mbox{ if }t\leq \tau , \\
& 0,\mbox{ if }t\geq 2\tau ,
\end{split}%
\right. \mbox{ and }\bar{\zeta}' \leq 0\mbox{ in } \mathbb{R}.
\end{equation*}
Now, if define $\overline{v}_{\varepsilon ,\eta }:=\overline{\zeta }%
\left\vert \nabla h_{\varepsilon ,\eta }\right\vert ^{2}$ and assume that $%
\overline{v}_{\varepsilon ,\eta }$ attains its maximum at $\left( 0,\bar{x}%
\right) $ for some $\bar{x}\in \Omega $, then we have
\begin{align*}
\overline{\zeta }\left( t\right) {\left\vert {\nabla {h_{\varepsilon ,\eta }}%
\left( {t,x}\right) }\right\vert ^{2}}& ={\overline{v}_{\varepsilon ,\eta }}%
\left( {t,x}\right) \leq {\overline{v}_{\varepsilon ,\eta }}\left( {0,\bar{x}%
}\right) ={\left\vert {\nabla {h_{\varepsilon ,\eta }}\left( {0,\bar{x}}%
\right) }\right\vert ^{2}}=\frac{1}{{{\gamma ^{2}}}}{\left( {{u_{0}}\left( {%
\bar{x}}\right) +\eta }\right) ^{2\left( {\frac{1}{\gamma }-1}\right) }}{%
\left\vert {\nabla {u_{0}}\left( {\bar{x}}\right) }\right\vert ^{2}} \\
& \leq {\left( {\frac{{{u_{0}}\left( {\bar{x}}\right) }}{{{u_{0}}\left( {%
\bar{x}}\right) +\eta }}}\right) ^{2\left( {1-\frac{1}{\gamma }}\right) }}{%
\left\Vert {\nabla u_{0}^{1/\gamma }}\right\Vert _{\infty }}\leq {\left\Vert
{\nabla u_{0}^{1/\gamma }}\right\Vert _{\infty }},
\end{align*}%
where we have used $\gamma \geq 1$ stemming from the additional assumption $%
\beta \leq 2-m$. Thus
\begin{equation*}
\left\vert {\nabla u_{\varepsilon ,\eta }^{1/\gamma }}\right\vert \leq {%
\left\Vert {\nabla u_{0}^{1/\gamma }}\right\Vert _{\infty }},\mbox{ in } \left(0,\infty\right)\times \Omega.
\end{equation*}%
Otherwise, $\overline{v}_{\varepsilon ,\eta }$ must attain its
maximum at some $\left( \bar{t}_{0},\bar{x}_{0}\right) \in \left( 0,2\tau
\right) \times \Omega $ since $\overline{v}_{\varepsilon ,\eta }=0$ on $%
\left\{ {\left( {2\tau ,\infty }\right) \times \Omega }\right\} \cup \left\{
{\left( {0,\infty }\right) \times \partial \Omega }\right\} $. Then,
repeating the proof of the first statement until \eqref{2.9}, and from the
fact that $\overline{\zeta }^{\prime }\leq 0$, we deduce
\begin{equation*}
\mathcal{B}{\overline{v}_{\varepsilon ,\eta }}=\mathcal{B}\overline{\zeta }{%
\left\vert {\nabla {h_{\varepsilon ,\eta }}}\right\vert ^{2}}\leq \left( {%
1+\beta -{\gamma ^{-1}}}\right) \overline{\zeta }{\psi _{\varepsilon }}%
\left( {h_{\varepsilon ,\eta }^{\gamma }}\right) ,\mbox{ at }\left( {{{\bar{t%
}}_{0}},{{\bar{x}}_{0}}}\right) .
\end{equation*}%
By the same argument, this leads us to
\begin{equation*}
\left\vert {\nabla u_{\varepsilon ,\eta }^{1/\gamma }(t,x)}\right\vert \leq {%
\left( {\frac{{2+\beta -m}}{{2\mathcal{B}}}}\right) ^{\frac{1}{2}}}, \mbox{ in } \left(0,\infty\right)\times \Omega.
\end{equation*}%
Then, combining both estimates we arrive to the conclusion.
\end{proof}
Now we will consider the one-dimensional case to prove similar gradient
estimates to the ones obtained in the above result. Moreover, we will get
also a gradient estimate for the case in which the diffusion dominates over
the absorption (similar to the one given in \cite{Kawohl1992b}).

\begin{lemma}
\label{lemma3} Let $N=1$, $m\geq 1$, $\beta \in \left( 0,m\right)$. Consider 
$u_{0}\in \mathcal{C}_{c}^{1}\left( \Omega \right)$, $u_{0}\geq 0$, $%
u_{0}\neq 0$ and $0<\eta \leq \min \left\{ {\varepsilon ,{{\left\Vert {u_{0}}%
\right\Vert }_{\infty }}}\right\}$. Then

\begin{itemize}
\item[i)] if $m<\beta +2$, there is a constant $C=C\left( m,\beta \right) $
such that 
\begin{equation*}
{\left\vert {\left( {u_{\varepsilon ,\eta }^{1/\gamma }}\right) _{x}\left( {%
t,x}\right) }\right\vert ^{2}}\leq C\left( {{t^{-1}}\left\Vert {u_{0}}%
\right\Vert _{{L^{\infty }}\left( \Omega \right) }^{1+\beta }+1}\right) ,%
\mbox{ in }\left( {0,\infty }\right) \times \Omega .
\end{equation*}%
In addition, if we assume \eqref{Hypo beta small} and ${\left( {{%
u_{0}^{1/\gamma }}}\right) ^{\prime }\in L^{\infty }}\left( \Omega \right) $%
, we get 
\begin{equation*}
\left\vert {\left( {u_{\varepsilon ,\eta }^{1/\gamma }}\right) }%
_{x}(t,x)\right\vert \leq \max \left\{ {\left\Vert \left( {{u_{0}^{1/\gamma }%
}}\right) ^{\prime }\right\Vert _{{L^{\infty }}\left( \Omega \right) },\frac{%
{m+\beta }}{\sqrt{2m\left( {m-\beta }\right) }}}\right\} ,\mbox{ in }\left[
0,\infty \right) \times \Omega .
\end{equation*}

\item[ii)] If $m\geq \beta +2$, then there is a constant $C=C\left( m\right) 
$ such that 
\begin{equation*}
{\left\vert {\left( {u_{\varepsilon ,\eta }^{m-1}}\right) _{x}\left( {t,x}%
\right) }\right\vert ^{2}}\leq C{t^{-1}}\left\Vert {u_{0}}\right\Vert _{{%
L^{\infty }}\left( \Omega \right) }^{m-1},\mbox{ in }\left( {0,\infty }%
\right) \times \Omega .
\end{equation*}
\end{itemize}
\end{lemma}

\begin{proof}
i) Repeating the proof of Lemma \ref{lemma2}
until \eqref{2.7} we get%
\begin{equation*}
\partial _{x}^{2}{h_{\varepsilon ,\eta }}=0\text{ and }{\partial _{t}}{%
v_{\varepsilon ,\eta }}-mh_{\varepsilon ,\eta }^{\gamma \left( {m-1}\right)
}\partial _{x}^{2}{v_{\varepsilon ,\eta }}\geq 0,\text{ at }\left( {{t_{0}},{%
x_{0}}}\right) .
\end{equation*}
Then
\begin{equation*}
\zeta {\partial _{x}}{h_{\varepsilon ,\eta }}\left( {{\partial _{tx}}{%
h_{\varepsilon ,\eta }}-mh_{\varepsilon ,\eta }^{\gamma \left( {m-1}\right)
}\partial _{x}^{3}{h_{\varepsilon ,\eta }}}\right) \geq -\frac{\zeta }{2}{%
\left( {{\partial _{x}}{h_{\varepsilon ,\eta }}}\right) ^{2}},\mbox{ at }%
\left( {{t_{0}},{x_{0}}}\right) .
\end{equation*}%
Combining this with the 1D-analogue of \eqref{h} and $\partial _{x}^{2}{%
h_{\varepsilon ,\eta }}\left( {{t_{0}},{x_{0}}}\right) =0$ we obtain
\begin{align*}
& m\left( {m\gamma -1}\right) \left( {1-\gamma \left( {m-1}\right) }\right)
\zeta h_{\varepsilon ,\eta }^{\gamma \left( {m-1}\right) -2}{\left( {{%
\partial _{x}}{h_{\varepsilon ,\eta }}}\right) ^{2}} \\
& \hspace{0.5cm}\leq \frac{{\zeta ^{\prime }}}{2}-\zeta {\psi _{\varepsilon }%
}^{\prime }\left( {h_{\varepsilon ,\eta }^{\gamma }}\right) h_{\varepsilon
,\eta }^{-\beta \gamma }+\left( {1+\beta -{\gamma ^{-1}}}\right) \zeta {\psi
_{\varepsilon }}\left( {h_{\varepsilon ,\eta }^{\gamma }}\right)
h_{\varepsilon ,\eta }^{-\gamma \left( {1+\beta }\right) },\mbox{ at }\left(
{{t_{0}},{x_{0}}}\right) .
\end{align*}%
Using the same argument, we arrive at the desired estimate.

ii) Let now $\overline{\gamma }:=1/\left( m-1\right) $ and define
$h_{\varepsilon ,\eta }:=u_{\varepsilon ,\eta }^{1/\overline{\gamma }}$.
Then, $h_{\varepsilon ,\eta }$ satisfies
\begin{equation*}
{\partial _{t}}{h_{\varepsilon ,\eta }}-m{h_{\varepsilon ,\eta }}{\partial
_{x}}{h_{\varepsilon ,\eta }}-\frac{m}{{m-1}}{\left( {{\partial _{x}}{%
h_{\varepsilon ,\eta }}}\right) ^{2}}+\left( {m-1}\right) {\psi
_{\varepsilon }}\left( {h_{\varepsilon ,\eta }^{\overline{\gamma }}}\right)
h_{\varepsilon ,\eta }^{1-\overline{\gamma }\left( {1+\beta }\right) }=0.
\end{equation*}%
As in \cite{Aronson1969} (see also \cite{Kawohl1992b} and \cite{Diaz1987b}),
we consider the auxiliary function $p\left( y\right) =N_{0}y\left(
4-y\right) /3$, for all $y\in \left[ 0,1\right] $, where ${N_{0}}:={\left( {2%
{{\left\Vert {{u_{0}}}\right\Vert }_{\infty }}}\right) ^{m-1}}$. Note that $%
p $ is invertible and
\begin{equation*}
p\in \left[ 0,N_{0}\right] ,\,p^{\prime }\in \left[ \frac{2N_{0}}{3},\frac{%
4N_{0}}{3}\right] ,\,p^{\prime \prime }=-\frac{2{N_{0}}}{3},\,\left( \frac{%
p^{\prime \prime }}{p^{\prime }}\right) ^{\prime }\leq -\frac{1}{4},%
\mbox{
in }\left[ {0,1}\right] .
\end{equation*}%
Its inverse function is given by ${p^{-1}}\left( z\right) =2-{\left( {%
4-3z/N_{0}}\right) ^{1/2}}$ for all $z\in \left[ 0,N_{0}\right] $. Finally,
\ define ${v_{\varepsilon ,\eta }}:={p^{-1}}\circ {h_{\varepsilon ,\eta }}$. We obtain the following equation, satisfied by $%
v_{\varepsilon ,\eta }$:
\begin{align}
{\partial _{t}}{v_{\varepsilon ,\eta }}& -mp\left( {{v_{\varepsilon ,\eta }}}%
\right) \partial _{x}^{2}{v_{\varepsilon ,\eta }}-\left( {\frac{m}{{m-1}}%
p^{\prime }+mp{{\left( {p^{\prime }}\right) }^{-1}}p^{\prime \prime }}%
\right) \left( {{v_{\varepsilon ,\eta }}}\right) {\left( {{\partial _{x}}{%
v_{\varepsilon ,\eta }}}\right) ^{2}}  \notag \\
& +\left( {m-1}\right) {\psi _{\varepsilon }}\left( {{p^{\overline{\gamma }}}%
}\right) {p^{1-\overline{\gamma }\left( {1+\beta }\right) }}{\left( {%
p^{\prime }}\right) ^{-1}}\left( {{v_{\varepsilon ,\eta }}}\right) =0,%
\mbox{
in }\left( {0,\infty }\right) \times \Omega .  \label{2.12}
\end{align}%
Differentiating in \eqref{2.12} with respect to the variable $x$, we obtain
\begin{align}
{\partial _{tx}}{v_{\varepsilon ,\eta }}-mp\left( {{v_{\varepsilon ,\eta }}}%
\right) \partial _{x}^{3}{v_{\varepsilon ,\eta }}=& \ mp^{\prime }\left( {{%
v_{\varepsilon ,\eta }}}\right) {\partial _{x}}{v_{\varepsilon ,\eta }}%
\partial _{x}^{2}{v_{\varepsilon ,\eta }}+\left( {\frac{m}{{m-1}}p^{\prime
}+mp{{\left( {p^{\prime }}\right) }^{-1}}p^{\prime \prime }}\right) ^{\prime
}\left( {{v_{\varepsilon ,\eta }}}\right) {\left( {{\partial _{x}}{%
v_{\varepsilon ,\eta }}}\right) ^{3}}  \notag \\
& +2\left( {\frac{m}{{m-1}}p^{\prime }+mp{{\left( {p^{\prime }}\right) }^{-1}%
}p^{\prime \prime }}\right) \left( {{v_{\varepsilon ,\eta }}}\right) {%
\partial _{x}}{v_{\varepsilon ,\eta }}\partial _{x}^{2}{v_{\varepsilon ,\eta
}}  \label{2.13} \\
& -\left( {m-1}\right) \left( {{\psi _{\varepsilon }}\left( {{p^{\overline{%
\gamma }}}}\right) {p^{1-\overline{\gamma }\left( {1+\beta }\right) }}{{%
\left( {p^{\prime }}\right) }^{-1}}}\right) ^{\prime }\left( {{%
v_{\varepsilon ,\eta }}}\right) {\partial _{x}}{v_{\varepsilon ,\eta }},%
\mbox{ in }\left( {0,\infty }\right) \times \Omega .  \notag
\end{align}%
Let us consider now the function ${w_{\varepsilon ,\eta }}:=\zeta {\left( {{%
\partial _{x}}{v_{\varepsilon ,\eta }}}\right) ^{2}}$ and use the same
argument as in the proof of Lemma \ref{lemma2}. Then, there is a point $%
\left( t_{0},x_{0}\right) \in \left( \tau /2,T+\tau /2\right) \times \Omega $
where $w_{\varepsilon ,\eta }$ attains its maximum and thus
\begin{equation*}
\partial _{x}^{2}{v_{\varepsilon ,\eta }}=0\mbox{ and }{\partial _{t}}{%
w_{\varepsilon ,\eta }}-mp\left( {{v_{\varepsilon ,\eta }}}\right) \partial
_{x}^{2}{w_{\varepsilon ,\eta }}\geq 0,\mbox{ at }\left( {{t_{0}},{x_{0}}}%
\right) .
\end{equation*}%
Then
\begin{equation*}
\zeta {\partial _{x}}{v_{\varepsilon ,\eta }}\left( {{\partial _{tx}}{%
v_{\varepsilon ,\eta }}-mp\left( {{v_{\varepsilon ,\eta }}}\right) \partial
_{x}^{3}{v_{\varepsilon ,\eta }}}\right) \geq -\frac{{\zeta ^{\prime }}}{2}{%
\left( {{\partial _{x}}{v_{\varepsilon ,\eta }}}\right) ^{2}},\mbox{ at }%
\left( {{t_{0}},{x_{0}}}\right) .
\end{equation*}%
Combining this and \eqref{2.13}, we get
\begin{align}
& -m\left( {\frac{m}{{m-1}}p^{\prime \prime }+p\left( {\frac{{p^{\prime
\prime }}}{{p^{\prime }}}}\right) ^{\prime }}\right) \left( {{v_{\varepsilon
,\eta }}}\right) \zeta {\left( {{\partial _{x}}{v_{\varepsilon ,\eta }}}%
\right) ^{2}}  \notag \\
& \hspace{5mm}\leq \frac{{\zeta ^{\prime }}}{2}-\zeta {\psi _{\varepsilon }}%
^{\prime }\left( {{p^{\overline{\gamma }}}}\right) {p^{-\beta \overline{%
\gamma }}}\left( {{v_{\varepsilon ,\eta }}}\right) +\left( {m-1}\right)
\zeta {\psi _{\varepsilon }}\left( {{p^{\overline{\gamma }}}}\right) {p^{1-%
\overline{\gamma }\left( {1+\beta }\right) }}{\left( {p^{\prime }}\right)
^{-2}}p^{\prime \prime }\left( {{v_{\varepsilon ,\eta }}}\right)  \notag \\
& \hspace{9mm}+\left( {\beta +2-m}\right) \zeta {\psi _{\varepsilon }}\left(
{{p^{\overline{\gamma }}}}\right) {p^{-\overline{\gamma }\left( {1+\beta }%
\right) }}\left( {{v_{\varepsilon ,\eta }}}\right) ,\mbox{ at }\left( {{t_{0}%
},{x_{0}}}\right) .  \label{2.15}
\end{align}%
Note that all the last three terms in the right hand side of \eqref{2.15} are
non-positive, and
\begin{equation*}
-m\left( {\frac{m}{{m-1}}p^{\prime \prime }+p\left( {\frac{{p^{\prime \prime
}}}{{p^{\prime }}}}\right) ^{\prime }}\right) \left( {{v_{\varepsilon ,\eta }%
}}\right) \geq \frac{{2{m^{2}}{N_{0}}}}{{3\left( {m-1}\right) }}+\frac{m}{4}%
p\left( {{v_{\varepsilon ,\eta }}}\right) \geq \frac{{2{m^{2}}{N_{0}}}}{{%
3\left( {m-1}\right) }}.
\end{equation*}%
Then \eqref{2.15} implies the following estimate
\begin{equation*}
\zeta {\left( {{\partial _{x}}{v_{\varepsilon ,\eta }}}\right) ^{2}}\left( {{%
t_{0}},{x_{0}}}\right) \leq \frac{{3{c_{0}}\left( {m-1}\right) }}{{4{m^{2}}{%
N_{0}}}}{\tau ^{-1}}.
\end{equation*}%
By using the same arguments than in Lemma \ref{lemma2}, the last
inequality implies
\begin{align*}
{\left( {{\partial _{x}}{h_{\varepsilon ,\eta }}}\right) ^{2}}\left( {\tau ,x%
}\right) & ={\left( {p^{\prime }}\right) ^{2}}\left( {{v_{\varepsilon ,\eta }%
}}\right) {\left( {{\partial _{x}}{v_{\varepsilon ,\eta }}}\right) ^{2}}%
\left( {\tau ,x}\right) \leq {\left( {\frac{{4{N_{0}}}}{3}}\right) ^{2}}%
\frac{{3{c_{0}}\left( {m-1}\right) }}{{4{m^{2}}{N_{0}}}}{\tau ^{-1}} \\
& =\frac{{{2^{m+1}}{c_{0}}\left( {m-1}\right) }}{{3{m^{2}}}}{\tau ^{-1}}%
\left\Vert {{u_{0}}}\right\Vert _{\infty }^{m-1},\ \forall x\in \Omega .
\end{align*}%
The rest of the proof is straightforward.
\end{proof}As in many other parabolic problems, the spatial gradient
estimates given in Lemma \ref{lemma2} imply the global $\mathcal{C}^{\alpha
} $-H\"{o}lder regularity of the solutions. Similar results hold for the
one-dimensional case by using Lemma \ref{lemma3}.

\begin{proposition}
\label{proposition1} Assume the conditions of the first part of Lemma \ref%
{lemma2}. Then, for any $\tau >0$, the following estimates hold for all $%
\left( t,x\right) ,\,\left( s,y\right) \in \left[ \tau ,\infty \right)
\times \Omega $:%
\begin{equation*}
\begin{array}{l}
\left\vert {u_{\varepsilon ,\eta }^{\frac{{m+1}}{2}}\left( {t,x}\right)
-u_{\varepsilon ,\eta }^{\frac{{m+1}}{2}}\left( {s,y}\right) }\right\vert
\leq C_{1}\left[ {C_{2}\left( {\left\vert {x-y}\right\vert +{{\left\vert {t-s%
}\right\vert }^{\frac{1}{{3N}}}}}\right) +C_{3}{{\left\vert {t-s}\right\vert 
}^{\frac{1}{3}}}}\right] , \\ 
C_{1}=C\left( {m,\beta ,N}\right) {\left( {{\tau ^{-1}}\left\Vert {u_{0}}%
\right\Vert _{{L^{\infty }}\left( \Omega \right) }^{1+\beta }+1}\right) ^{%
\frac{1}{2}},}\text{ }{C}_{2}={\left\Vert {u_{0}}\right\Vert _{{L^{\infty }}%
\left( \Omega \right) }^{\frac{{1-\beta }}{2}}}\text{,} \\ 
{C}_{3}={{{\left\vert \Omega \right\vert }^{\frac{1}{2}}}\left\Vert {u_{0}}%
\right\Vert _{{L^{\infty }}\left( \Omega \right) }^{\frac{{m-\beta }}{2}}}%
\end{array}%
\end{equation*}%
if $\beta \leq 1$, and%
\begin{equation*}
\begin{array}{l}
\left\vert {u_{\varepsilon ,\eta }^{\frac{{m+1}}{2}}\left( {t,x}\right)
-u_{\varepsilon ,\eta }^{\frac{{m+1}}{2}}\left( {s,y}\right) }\right\vert
\leq \,\widehat{C_{1}}\left[ \widehat{{C_{2}}}{\left( {\left\vert {x-y}%
\right\vert +{{\left\vert {t-s}\right\vert }^{\frac{1}{{3N}}}}}\right) }^{%
\frac{m+1}{m+\beta }}{+{C}_{3}{{\left\vert {t-s}\right\vert }^{\frac{1}{3}}}}%
\right] \\ 
\widehat{C_{1}}=C(m,\beta ,{\left\Vert {{u_{0}}}\right\Vert _{{L^{\infty }}%
\left( \Omega \right) })}\text{,} \\ 
\widehat{{C_{2}}}=2{\left( {{\tau ^{-1}}\left\Vert {u_{0}}\right\Vert _{{%
L^{\infty }}\left( \Omega \right) }^{1+\beta }+1}\right) ^{\frac{m+1}{%
2(m+\beta )}}},%
\end{array}%
\end{equation*}%
if $\beta >1$. Moreover, if $\beta +m\leq 2$ and $\nabla u_{0}^{1/\gamma
}\in L^{\infty }\left( \Omega \right) $, then%
\begin{equation*}
\begin{array}{c}
\left\vert {u_{\varepsilon ,\eta }^{\frac{{m+1}}{2}}\left( {t,x}\right)
-u_{\varepsilon ,\eta }^{\frac{{m+1}}{2}}\left( {s,y}\right) }\right\vert
\leq K_{1}\left[ {\left( {\left\vert {x-y}\right\vert +{{\left\vert {t-s}%
\right\vert }^{\frac{1}{{3N}}}}}\right) +K}_{2}{{{\left\vert {t-s}%
\right\vert }^{\frac{1}{3}}}}\right] ,\, \\ 
K_{1}=3\cdot {2^{\frac{{1-\beta }}{2}}}\frac{{m+1}}{{m+\beta }}\left\Vert {%
u_{0}}\right\Vert _{{L^{\infty }}\left( \Omega \right) }^{\frac{{1-\beta }}{2%
}}\max \left\{ {{{\left\Vert {\nabla u_{0}^{1/\gamma }}\right\Vert }_{{%
L^{\infty }}\left( \Omega \right) }},{{\left[ {\frac{{\left( {2+\beta -m}%
\right) {{\left( {m+\beta }\right) }^{2}}}}{{2m\left( {{\Delta _{m,N}}-{{%
\left( {\beta +1-m}\right) }^{2}}}\right) }}}\right] }^{1/2}}}\right\} \\ 
{K}_{2}=C\left( {m,\beta ,N}\right) {\left\vert \Omega \right\vert ^{\frac{1%
}{2}}}{\left( {{\tau ^{-1}}\left\Vert {u_{0}}\right\Vert _{{L^{\infty }}%
\left( \Omega \right) }^{1+\beta }+1}\right) ^{\frac{1}{2}}}\left\Vert {u_{0}%
}\right\Vert _{{L^{\infty }}\left( \Omega \right) }^{\frac{{m-\beta }}{2}},%
\end{array}%
\end{equation*}%
for all $\left( t,x\right) ,\,\left( s,y\right) \in \left[ 0,\infty \right)
\times \Omega $.
\end{proposition}

\begin{proof}
Let us first extend $u_{\varepsilon ,\eta }$ by $\eta $ outside $\Omega $ if
needed and denote still by $u_{\varepsilon ,\eta }$ to that extension. For
arbitrary $t\geq s\geq \tau >0$, by multiplying the equation  by ${\partial
_{t}}u_{\varepsilon ,\eta }^{m}=mu_{\varepsilon ,\eta }^{m-1}{\partial _{t}}{%
u_{\varepsilon ,\eta }}$ and integrating by parts over $\left( s,t\right)
\times \Omega $ we get
\begin{equation*}
\int_{s}^{t}{\int_{\Omega }{mu_{\varepsilon ,\eta }^{m-1}{{\left\vert {{%
\partial _{t}}{u_{\varepsilon ,\eta }}}\right\vert }^{2}}dxd\sigma }}+\frac{1%
}{2}\frac{d}{{dt}}\int_{s}^{t}{\int_{\Omega }{{{\left\vert {\nabla
u_{\varepsilon ,\eta }^{m}}\right\vert }^{2}}dxd\sigma }}+\int_{s}^{t}{%
\int_{\Omega }{mu_{\varepsilon ,\eta }^{m-1}{g_{\varepsilon }}\left( {{%
u_{\varepsilon ,\eta }}}\right) {\partial _{t}}{u_{\varepsilon ,\eta }}%
dxd\sigma }}=0.
\end{equation*}%
Define ${G_{\varepsilon }}\left( r\right) :=m\int_{0}^{r}{{s^{m-1}}{%
g_{\varepsilon }}\left( s\right) ds}$. Notice that
\begin{equation*}
{G_{\varepsilon }}\left( r\right) \leq m\int_{0}^{r}{{s^{m-\beta -1}}ds}=%
\frac{m}{{m-\beta }}{r^{m-\beta }},\ \forall r>0.
\end{equation*}%
Then the last equality implies that
\begin{equation*}
\int_{s}^{t}{\int_{\Omega }{mu_{\varepsilon ,\eta }^{m-1}{{\left\vert {{%
\partial _{t}}{u_{\varepsilon ,\eta }}}\right\vert }^{2}}dxd\sigma }}\leq
\frac{1}{2}\int_{\Omega }{{{\left\vert {\nabla u_{\varepsilon ,\eta
}^{m}\left( {s,x}\right) }\right\vert }^{2}}dx}+\int_{\Omega }{{%
G_{\varepsilon }}\left( {{u_{\varepsilon ,\eta }}\left( {s,x}\right) }%
\right) dx}.
\end{equation*}%
Let ${z_{\varepsilon ,\eta }}:=2\sqrt{m}u_{\varepsilon ,\eta }^{\left( {m+1}%
\right) /2}/\left( {m+1}\right) $. Using \eqref{GE} we get
\begin{align*}
\int_{s}^{t}{\int_{\Omega }{{{\left\vert {{\partial _{t}}{z_{\varepsilon
,\eta }}}\right\vert }^{2}}dxd\sigma }}& \leq C\left( {m,\beta ,N}\right)
\left( {{\tau ^{-1}}\left\Vert {{u_{0}}}\right\Vert _{\infty }^{1+\beta }+1}%
\right) \int_{\Omega }{{u_{\varepsilon ,\eta }^{m-\beta }}\left( {s,x}%
\right) dx} \\
& \leq C\left( {m,\beta ,N}\right) \left\vert \Omega \right\vert \left( {{%
\tau ^{-1}}\left\Vert {{u_{0}}}\right\Vert _{\infty }^{1+\beta }+1}\right)
\left\Vert {{u_{0}}}\right\Vert _{\infty }^{m-\beta }=:{C_{0}}.
\end{align*}%
Given $x,\,y\in \Omega $, define $r:=\left\vert {x-y}\right\vert +{%
\left\vert {t-s}\right\vert ^{\frac{1}{{3N}}}}$. Then, for some $\bar{x}\in
B_{r}\left( x\right) $:
\begin{align*}
{\left\vert {{z_{\varepsilon ,\eta }}\left( {t,\bar{x}}\right) -{%
z_{\varepsilon ,\eta }}\left( {s,\bar{x}}\right) }\right\vert ^{2}}& \leq
\left( {t-s}\right) \int_{s}^{t}{{{\left\vert {{\partial _{t}}{%
z_{\varepsilon ,\eta }}\left( {\sigma ,\bar{x}}\right) }\right\vert }^{2}}%
d\sigma } \\
& =\frac{{t-s}}{{\left\vert {{B_{r}}}\right\vert }}\int_{s}^{t}{\int_{{B_{r}}%
\left( x\right) }{{{\left\vert {{\partial _{t}}{z_{\varepsilon ,\eta }}%
\left( {\sigma ,z}\right) }\right\vert }^{2}}dzd\sigma }}\leq \frac{{{C_{0}}%
\left\vert {t-s}\right\vert }}{{{\alpha _{N}}{r^{N}}}}\leq \frac{{{C_{0}}{{%
\left\vert {t-s}\right\vert }^{\frac{2}{3}}}}}{{{\alpha _{N}}}},
\end{align*}%
where ${\alpha _{N}}:=\left\vert {{B_{1}}}\right\vert =2{\pi ^{N/2}}/\left( {%
N\Gamma \left( {N/2}\right) }\right) $. From the triangle inequality one has
\begin{align*}
& \left\vert {{z_{\varepsilon ,\eta }}\left( {t,x}\right) -{z_{\varepsilon
,\eta }}\left( {s,y}\right) }\right\vert \\
& \hspace{1cm}\leq \left\vert {{z_{\varepsilon ,\eta }}\left( {t,x}\right) -{%
z_{\varepsilon ,\eta }}\left( {t,\bar{x}}\right) }\right\vert +\left\vert {{%
z_{\varepsilon ,\eta }}\left( {t,\bar{x}}\right) -{z_{\varepsilon ,\eta }}%
\left( {s,\bar{x}}\right) }\right\vert +\left\vert {{z_{\varepsilon ,\eta }}%
\left( {s,\bar{x}}\right) -{z_{\varepsilon ,\eta }}\left( {s,y}\right) }%
\right\vert .
\end{align*}%
Then, if $\beta \leq 1$,%
\begin{equation*}
\begin{array}{cc}
\left\vert {{z_{\varepsilon ,\eta }}\left( {t,x}\right) -{z_{\varepsilon
,\eta }}\left( {s,y}\right) }\right\vert &  \\
\leq {\left\Vert {\nabla {z_{\varepsilon ,\eta }}\left( t\right) }%
\right\Vert _{\infty }}\left\vert {x-\bar{x}}\right\vert +{\left( {\frac{{{%
C_{0}}}}{{{\alpha _{N}}}}}\right) ^{1/2}}{\left\vert {t-s}\right\vert ^{1/3}}%
+{\left\Vert {\nabla {z_{\varepsilon ,\eta }}\left( s\right) }\right\Vert
_{\infty }}\left\vert {\bar{x}-y}\right\vert . &  \\
&
\end{array}%
\end{equation*}%
Combining this with the estimate
\begin{equation*}
\left\vert {\nabla {z_{\varepsilon ,\eta }}\left( {t,x}\right) }\right\vert =%
\sqrt{m}u_{\varepsilon ,\eta }^{\frac{{m-1}}{2}}\left( {t,x}\right)
\left\vert {\nabla {u_{\varepsilon ,\eta }}\left( {t,x}\right) }\right\vert
\leq C\left( {m,\beta ,N}\right) u_{\varepsilon ,\eta }^{\frac{{1-\beta }}{2}%
}\left( {t,x}\right) {\left( {{t^{-1}}\left\Vert {{u_{0}}}\right\Vert _{{%
L^{\infty }}\left( \Omega \right) }^{1+\beta }+1}\right) ^{\frac{1}{2}}},
\end{equation*}%
we get the first desired estimate.

\noindent If $\beta >1$, then, since ${{z_{\varepsilon ,\eta }}\left( {t,x}%
\right) =C(m,\beta )}\left( u_{\varepsilon ,\eta }^{\frac{{m+\beta }}{2}%
}\right) ^{\nu }$ with $\nu =(m+1)/(m+\beta )$ and $\nu \in (0,1)$, using
the H\"{o}lder continuity of the function $r\rightarrow r^{\nu }$ we get
\begin{equation*}
\begin{array}{c}
\left\vert {{z_{\varepsilon ,\eta }}\left( {t,x}\right) -{z_{\varepsilon
,\eta }}\left( {t,\bar{x}}\right) }\right\vert  \\
\leq C(m,\beta ,{\left\Vert {{u_{0}}}\right\Vert _{{L^{\infty }}\left(
\Omega \right) })}\left\vert u_{\varepsilon ,\eta }^{\frac{{m+\beta }}{2}}{%
\left( {t,x}\right) -u_{\varepsilon ,\eta }^{\frac{{m+\beta }}{2}}\left( {t,%
\bar{x}}\right) }\right\vert ^{\nu } \\
\leq C(m,\beta ,{\left\Vert {{u_{0}}}\right\Vert _{{L^{\infty }}\left(
\Omega \right) })\left\Vert {\nabla u_{\varepsilon ,\eta }^{\frac{{m+\beta }%
}{2}}\left( t\right) }\right\Vert _{\infty }^{\nu }}\left\vert {x-\bar{x}}%
\right\vert ^{\nu }%
\end{array}%
\end{equation*}%
and we argue analougously with the term $\left\vert {{z_{\varepsilon ,\eta }}%
\left( {s,\bar{x}}\right) -{z_{\varepsilon ,\eta }}\left( {s,y}\right) }%
\right\vert $ to get the desired estimate.

\noindent The proof of the remaining statement can be obtained easily by
using \eqref{GE1} instead of \eqref{GE} in the last inequality. Note also
that $\beta \leq 2-m<1$. This completes our proof. 
\end{proof}

\bigskip

Before ending this section we point out that the estimates \eqref{GE} and %
\eqref{GE1} are independent of $\varepsilon $ and $\eta $. Thus, they play a
role of some useful \textbf{a priori estimates} which will allow the passing
to the limit as $\eta ,\,\varepsilon \downarrow 0$, successively. So, for
any $\varepsilon >0$ fixed, since $g_{\varepsilon }\left( s\right) $ is a
globally Lipschitz function, we can pass to the limit as $\eta \downarrow 0$
showing that $u_{\varepsilon ,\eta }\rightarrow u_{\varepsilon }$ and that $%
u_{\varepsilon }$ is the (unique) weak solution of the problem: 
\begin{equation*}
(P_{\varepsilon })\left\{ 
\begin{split}
& {\partial _{t}}u-\Delta {u^{m}}+{g_{\varepsilon }}\left( u\right) =0,%
\mbox{ in }\left( {0,\infty }\right) \times \Omega , \\
& u=0,\mbox{ on }\left( {0,\infty }\right) \times \partial \Omega , \\
& u\left( {0,x}\right) ={u_{0,\varepsilon }}\left( x\right) ,\mbox{ in }%
\Omega ,
\end{split}%
\right.
\end{equation*}%
where, more in general, we can assume that the initial datum is also
depending on the parameter $\varepsilon >0$, with $u_{0,\varepsilon }\in
L^{\infty }\left( \Omega \right) $, $u_{0,\varepsilon }\geq 0$ (see details,
e.g., in \cite{Aronson1982} or \cite{Vazquez2007}). Moreover, obviously $%
u_{\varepsilon }$ also satisfies the corresponding pointwise gradient
estimates given in Lemma \ref{lemma2} and Lemma \ref{lemma3}.

In the following section we will justify that the limit $\varepsilon
\downarrow 0$ allows us to prove the existence of solutions of equation %
\eqref{P} presented in Theorem \ref{theorem1}.

\section{Proof of Theorem \protect\ref{theorem1} and study of the Cauchy
problem}

In order to complete the proof of Theorem \ref{theorem1} we will structure
it in a series of steps.

\noindent \textbf{Step 1: Monotone convergence in $L^1\left(0,T; L_\delta
^1\left(\Omega\right)\right)$ for bounded initial data.}

\noindent Let us first consider the case in which $u_{0}=u_{0,\varepsilon
}\in L^{\infty }\left( \Omega \right) $, $u_{0}\geq 0$. The family of
functions, $\left( u_{\varepsilon }\right) _{\varepsilon >0}$, obtained at
the end of the previous section, forms a bounded monotone sequence. Indeed,
from the definition of $g_{\varepsilon }$ we see that 
\begin{equation*}
g_{\varepsilon _{1}}\left( s\right) \geq g_{\varepsilon _{2}}\left( s\right)
,\ \forall s\in \mathbb{R},\mbox{ for }0<\varepsilon _{1}<\varepsilon _{2}.
\end{equation*}%
This implies that $u_{\varepsilon _{1}}$ is a subsolution of the equation
satisfied by $u_{\varepsilon _{2}}$ and then since the comparison principle
holds for the problem $(P_{\varepsilon })$ (see e.g., \cite{Aronson1982}) we
get that 
\begin{equation*}
u_{\varepsilon _{1}}\leq u_{\varepsilon _{2}},\mbox{ in }\left( 0,\infty
\right) \times \Omega ,\mbox{ for }0<\varepsilon _{1}<\varepsilon _{2}.
\end{equation*}%
Then, there is a nonnegative function $u\in L^{1}\left( 0,T;L_{\delta
}^{1}\left( \Omega \right) \right) $ such that 
\begin{equation*}
u_{\varepsilon }\downarrow u,\mbox{ as }\varepsilon \downarrow 0.
\end{equation*}%
From the $L_{\delta }^{1}\left( \Omega \right) $-contractivity proved in
Section 6.6 of \cite{Vazquez2007} we know that for all $T\in \left( 0,\infty
\right) $, 
\begin{equation*}
\int_{\Omega }{{u_{\varepsilon }}\left( {T,x}\right) \zeta \left( x\right) dx%
}+\int_{0}^{T}{\int_{\Omega }{{g_{\varepsilon }}\left( {u_{\varepsilon }}%
\right) \zeta \left( x\right) dxdt}}\leq \int_{\Omega }{{u_{0}}\left(
x\right) \zeta \left( x\right) dx}.
\end{equation*}%
It follows from the last inequality and the Dominated Convergence Theorem
that there is a function $\Upsilon $ such that 
\begin{equation*}
\lim_{\varepsilon \downarrow 0}{g_{\varepsilon }}\left( {u_{\varepsilon }}%
\right) =\Upsilon ,\mbox{ in }L^{1}\left( 0,T;L_{\delta }^{1}\left( \Omega
\right) \right) .
\end{equation*}%
Moreover, the monotonicity of $\left( u_{\varepsilon }\right) _{\varepsilon
>0}$ implies 
\begin{equation*}
{g_{\varepsilon }(u_{\varepsilon }}\left( {t,x}\right) )\geq {g_{\varepsilon
}}\left( {u_{\varepsilon }}\right) {\chi _{\left\{ {u>0}\right\} }}\left( {%
t,x}\right) ,\mbox{ a.e. in }\left( {0,\infty }\right) \times \Omega ,
\end{equation*}%
so 
\begin{equation}
\lim_{\varepsilon \downarrow 0}{g_{\varepsilon }(u_{\varepsilon }}\left( {t,x%
}\right) )=\Upsilon \left( {t,x}\right) \geq {u^{-\beta }}{\chi _{\left\{ {%
u>0}\right\} }}\left( {t,x}\right) ,\mbox{ a.e. in }\left( {0,\infty }%
\right) \times \Omega .  \label{3.2}
\end{equation}%
Thus, 
\begin{equation*}
{\left\Vert {{u^{-\beta }}{\chi _{\left\{ {u>0}\right\} }}}\right\Vert
_{L^{1}\left( 0,T;L_{\delta }^{1}\left( \Omega \right) \right) }}\leq
\int_{\Omega }{{u_{0}}\left( x\right) \zeta (x)dx}.
\end{equation*}%
As a matter of fact, we will prove later that 
\begin{equation}
\Upsilon ={u^{-\beta }}{\chi _{\left\{ {u>0}\right\} }},\mbox{ in }%
L^{1}\left( 0,T;L_{\delta }^{1}\left( \Omega \right) \right) .  \label{3.3}
\end{equation}

\noindent \textbf{Step 2: Passing to the limit in $\mathcal{C}\left( \left[
0,T\right] ;L^{1}\left( \Omega \right) \right) $ and $\mathcal{C}\left( %
\left[ 0,T\right] ;L_{\delta }^{1}\left( \Omega \right) \right) $ for
bounded initial data.}

\noindent Let us start by presenting some arguments which are valid to the
case in which $u_0\in L^1\left(\Omega\right)$, $u_0\ge 0$. Since $%
u_\varepsilon$ are limits of classical solutions, by applying Section 3 of
Benilan, Crandall and Sacks \cite{Benilan-Crandall-Sacks}, we know that $%
\left(u_{\varepsilon }\right) _{\varepsilon >0}$ are generalized (and $L^{1}$%
-mild) solutions of the problems 
\begin{equation}
\left\{ 
\begin{split}
& {\partial _{t}}u-\Delta {u^{m}}={f_{\varepsilon }},\mbox{ in }\left( {0,T}%
\right) \times \Omega , \\
& u=0,\mbox{ on }\left( {0,T}\right) \times \partial \Omega , \\
& u\left( {0,x}\right) ={u_{0,\varepsilon }}\left( x\right) ,\mbox{ in }%
\Omega ,
\end{split}%
\right.  \label{3.4}
\end{equation}%
with $f_{\varepsilon }\in L^{1}\left( 0,T;L^{1}\left( \Omega \right) \right) 
$ given by $f_{\varepsilon }\left( t,x\right) =-g_{\varepsilon }\left(
u_{\varepsilon }\left( t,x\right) \right)$.

From the Step 1 we know that $f_{\varepsilon }\rightarrow -\Upsilon $ in $%
L^{1}\left( 0,T;L^{1}\left( \Omega \right) \right) $ and $u_{0,\varepsilon
}\rightarrow u_{0}$ in $L^{1}\left( \Omega \right) $, as $\varepsilon
\downarrow 0$. Then, by \cite[Theorem I]{Benilan-Crandall-Sacks} we know
that $u_{\varepsilon }\rightarrow u$ in $\mathcal{C}\left( \left[ 0,T\right]
;L^{1}\left( \Omega \right) \right) $ with $u$ the unique generalized (and $%
L^{1}$-mild) solution of the problem 
\begin{equation}
\left\{ 
\begin{split}
& {\partial _{t}}u-\Delta {u^{m}}=-\Upsilon ,\mbox{ in }\left( {0,T}\right)
\times \Omega , \\
& u=0,\mbox{ on }\left( {0,T}\right) \times \partial \Omega , \\
& u\left( {0,x}\right) ={u_{0}}\left( x\right) ,\mbox{ in }\Omega .
\end{split}%
\right.  \label{3.5}
\end{equation}%
Let us now prove \eqref{3.3}. Since $u_{\varepsilon }$ is a weak solution of
equation $(P_{\varepsilon })$, one has 
\begin{equation*}
\iint_{\mathrm{Supp}\left( \varphi \right) }{\left( {\ -{u_{\varepsilon }}{%
\partial _{t}}\varphi -u_{\varepsilon }^{m}\Delta \varphi +{g_{\varepsilon }}%
\left( {u_{\varepsilon }}\right) \varphi }\right) dxdt}=0,\ \forall \varphi
\in \mathcal{C}_{c}^{\infty }\left( \left( 0,T\right) \times \Omega \right)
,\,\varphi \geq 0.
\end{equation*}%
Letting $\varepsilon \downarrow 0$ and since $u$ is also a very weak
solution of problem (\ref{3.5}), we get 
\begin{equation*}
-\iint_{\mathrm{Supp}\left( \varphi \right) }{\left( {\ u{\partial _{t}}%
\varphi +u^{m}\Delta \varphi }\right) dxdt}+\mathop {\lim }%
\limits_{\varepsilon \downarrow 0}\iint_{\mathrm{Supp}\left( \varphi \right)
}{{g_{\varepsilon }}\left( {u_{\varepsilon }}\right) \varphi dxdt}=0.
\end{equation*}%
Thus, 
\begin{equation}
\mathop {\lim }\limits_{\varepsilon \downarrow 0}\iint_{\mathrm{Supp}\left(
\varphi \right) }{{g_{\varepsilon }}\left( {u_{\varepsilon }}\right) \varphi
dxdt}=\iint_{\mathrm{Supp}\left( \varphi \right) }{{u^{-\beta }}{\chi
_{\left\{ {u>0}\right\} }}\varphi dxdt},\ \forall \varphi \in C_{c}^{\infty
}\left( {\left( {0,T}\right) \times \Omega }\right) ,\,\varphi \geq 0.
\label{3.6}
\end{equation}%
Then, $\Upsilon ={u^{-\beta }}{\chi _{\left\{ {u>0}\right\} }}$, in $%
L^{1}\left( 0,T;L^{1}\left( \Omega \right) \right) $ follows from \eqref{3.2}
and \eqref{3.6}.

The same conclusion also holds for similar arguments for the more general
case in which $u_{0}\in L_{\delta }^{1}\left( \Omega \right) $, $u_{0}\geq 0$%
. The only modification to be justified is the application of the continuous
dependence result for mild-solutions of \eqref{3.4}. The mean ingredient of
the proof of Theorem I of \cite{Benilan-Crandall-Sacks} is that the abstract
operator associated to problem $(P_{\varepsilon })$ is a $m$-$T$-accretive
operator on the Banach space $X=L^{1}\left( \Omega \right) $ but the same
conclusion arises once we prove the same properties on the space $X=L_{\zeta
}^{1}\left( \Omega \right) =L_{\delta }^{1}\left( \Omega \right) $ (with $%
\zeta $ given by \eqref{Eq Function xi}). This is more or less implicitly
well-known property (see, e.g., Section 6.6 of \cite{Vazquez2007}) but since
we are unable to find a more detailed proof we will get here a short proof
of this set of properties. Given $f\in L_{\delta }^{1}\left( \Omega \right) $
and $\lambda \geq 0$, we start by recalling the definition of very weak
solution of the stationary problem 
\begin{equation}
P(f,\lambda )=\left\{ 
\begin{split}
-\Delta (\left\vert {u}\right\vert {^{m-1}u)+\lambda u}={f},& \text{ in }%
\Omega , \\
\left\vert {u}\right\vert {^{m-1}u}=0,& \text{ on }\partial \Omega .
\end{split}%
\right.  \label{Ec. Porous media very weak}
\end{equation}

\begin{definition}
\label{definition4} Given $f\in L_\delta ^1\left(\Omega\right)$ and $\lambda
\ge 0$, a function $u\in L_\delta ^1\left(\Omega\right)$ is called a \emph{%
very weak solution} of $P\left(f,\lambda\right)$ if $\left|u\right|^{m-1} u
\in L^1\left(\Omega\right)$ and for any $\psi \in W^{2,\infty}
\left(\Omega\right)\cap W_0^{1,\infty}\left(\Omega\right)$, 
\begin{align*}
\int_\Omega {{u^m}\left( x \right)\Delta \psi \left( x \right)dx} + \lambda
\int_\Omega {u\left( x \right)\psi \left( x \right)dx} = \int_\Omega {%
f\left( x \right)\psi \left( x \right)dx} .
\end{align*}
\end{definition}

We have

\begin{lemma}
\label{lemma4} Let $X=L_{\zeta }^{1}\left( \Omega \right) $, $m>0$ and
define the operator $A:D\left( A\right) \rightarrow X$ given by 
\begin{equation*}
Au=-\Delta (\left\vert {u}\right\vert {^{m-1}u)}=:f,\ u\in D\left( A\right) ,
\end{equation*}%
with 
\begin{equation*}
D\left( A\right) =\left\{ u\in L_{\zeta }^{1}\left( \Omega \right) ;u%
\mbox{ is a
very weak solution of }P\left( f,0\right) \mbox{ for some }f\in L_{\zeta
}^{1}\left( \Omega \right) \right\} .
\end{equation*}%
Then $A$ is a $m$-$T$-accretive operator on the Banach space $X$ and $%
\overline{D\left( A\right) }=X$.
\end{lemma}

\begin{proof}
To show that $A$ is a $T$-accretive operator on $X$ we have to show that,
given $f,\,\widehat{f}\in L_{\zeta }^{1}\left( \Omega \right) $ and $\lambda
>0$, if $u,\,\widehat{u}$ are very weak solutions of $P\left( f,\lambda
\right) $ and $P\left( \widehat{f},\lambda \right) $, respectively. Then
\begin{equation}
\lambda \left\Vert \left[ u-\widehat{u}\right] _{+}\right\Vert _{L_{\zeta
}^{1}\left( \Omega \right) }\leq \left\Vert \left[ f-\widehat{f}\right]
_{+}\right\Vert _{L_{\zeta }^{1}\left( \Omega \right) }.  \label{accretivity}
\end{equation}%
But by introducing $v=\left\vert u\right\vert ^{m-1}u$ then $v\in
L^{1}\left( \Omega \right) $ is a very weak solution of
\begin{equation}
\widetilde{P}\left( {f,\lambda }\right) =\left\{
\begin{split}
-\Delta v+\lambda {{\left\vert v\right\vert }^{\frac{1}{m}-1}}v=f,& \mbox{
in }\Omega , \\
v=0,& \mbox{ on }\partial \Omega ,
\end{split}%
\right.
\end{equation}%
(and similarly for $\widehat{v}=\left\vert \widehat{u}\right\vert ^{m-1}%
\widehat{u}$). Assume for the moment that $f,\,\widehat{f}\geq 0$ and thus
the positivity of $u,\,\widehat{u}$ was proved in \cite{Brezis 1971} (see
also \cite{Brezis 1996}) and the estimate \eqref{accretivity} coincides
exactly with the estimate (19) given in Theorem 2.5 of D\'{\i}az and
Rakotoson \cite{Diaz2010} (notice that although $L_{\zeta }^{1}\left( \Omega
\right) =L_{\delta }^{1}\left( \Omega \right) $, thanks to (\ref{(1.14)}),
the norms $\left\Vert \cdot \right\Vert _{L_{\zeta }^{1}\left( \Omega
\right) }$ and $\left\Vert \cdot \right\Vert _{L_{\delta }^{1}\left( \Omega
\right) }$ are related by some constants: by replacing $\left\Vert \cdot
\right\Vert _{L_{\delta }^{1}\left( \Omega \right) }$ by the norm $%
\left\Vert \cdot \right\Vert _{L_{\zeta }^{1}\left( \Omega \right) }$ then
the constant $C$ arising in the estimate (19) given in Theorem 2.5 of D\'{\i}%
az and Rakotoson \cite{Diaz2010} becomes exactly $C=1$ as needed in %
\eqref{accretivity}). By using the decomposition $f=f_{+}-f_{-}$ the
estimate \eqref{accretivity} holds for general $f,\,\widehat{f}\in L_{\zeta
}^{1}\left( \Omega \right) $. An alternative proof can be obtained by
applying the \textit{local Kato's inequality} given in Theorem 4.4 of \cite%
{Diaz2018}.

The proof of the $m$-accretivity of $A$ (i.e., $R\left(A + \lambda I\right)
=X$) was already proved in \cite{Brezis 1971} (see also \cite{Brezis 1996}
and Theorem 2.5 of \cite{Diaz2010}).

Moreover, given $f\in L_\zeta ^1\left(\Omega\right)$ we consider $u_\alpha
\in D\left(A\right)$ be the unique solution of $\alpha A u_\alpha + u
_\alpha = f$. Then making $\alpha \downarrow 0$ we have (again by Theorem
2.5 \cite{Diaz2010}) that $u_\alpha \to f$ in $L_\zeta ^1 \left(\Omega\right)
$, which proves that $\overline{D\left(A\right)} = X$.
\end{proof}

As a consequence of Lemma \ref{lemma4}, we can apply the Crandall-Liggett
theorem and by the accretive operator theory we know that $f_{\varepsilon
}\rightarrow -\Upsilon $ in $L^{1}\left( 0,T;L_{\zeta }^{1}\left( \Omega
\right) \right) $ and $u_{0,\varepsilon }\rightarrow u_{0}$ in $L_{\zeta
}^{1}\left( \Omega \right) $, implies that $u_{\varepsilon }\rightarrow u$
in $\mathcal{C}\left( \left[ 0,T\right] ;L_{\zeta }^{1}\left( \Omega \right)
\right) $ with $u_{\varepsilon }$ and $u$ the unique $L_{\zeta }^{1}\left(
\Omega \right) $-mild solutions of the problems \eqref{3.4} and \eqref{3.5},
respectively, as $\varepsilon \downarrow 0$. Now, the adaptation of the
proof of \cite[Theorem I]{Benilan-Crandall-Sacks} to show that $%
u_{\varepsilon }\rightarrow u$ in $\mathcal{C}\left( \left[ 0,T\right]
;L_{\zeta }^{1}\left( \Omega \right) \right) $ as generalized solutions is a
trivial fact. This implies, as before, that $\Upsilon ={u^{-\beta }}{\chi
_{\left\{ {u>0}\right\} }}$, in $L^{1}\left( 0,T;L_{\delta }^{1}\left(
\Omega \right) \right) $.

\begin{remark}
We point out that the uniqueness of a generalized (or $L^1$-mild) solution
of the problem \eqref{3.5}, when $\Upsilon \left(t,x\right)$ is prescribed
in $L^1\left(0,T; L_\zeta ^1\left(\Omega\right)\right)$ does not imply the
uniqueness of the generalized (or $L_\zeta ^1$-mild) solution of the
non-monotone problem \eqref{P}. This question remains as an open problem: as
in \cite{Winkler2007}, the uniqueness of solutions fails even for general
bounded nonnegative initial data. Some partial results are given in \cite%
{Diaz-Giacomoni}.
\end{remark}

\noindent \textbf{Step 3: Maximality of the above constructed solution.} Let
us show that if $v$ is a different solution of equation \eqref{P} then, 
\begin{equation*}
v\left( t,x\right) \leq u\left( t,x\right) ,\mbox{ a.e. in }\left( 0,\infty
\right) \times \Omega .
\end{equation*}%
Indeed, since $g_{\varepsilon }\left( v\right) \leq v^{-\beta }\chi
_{\left\{ v>0\right\} }$, $\forall \varepsilon >0$, then 
\begin{equation*}
\partial _{t}v-\Delta v^{m}+g_{\varepsilon }\left( v\right) \leq 0,%
\mbox{ in
}\mathcal{D}^{\prime }\left( \left( 0,\infty \right) \times \Omega \right) ,
\end{equation*}%
which implies that $v$ is a subsolution of problem $(P_{\varepsilon })$
(with the same initial datum). Since $g_{\varepsilon }\left( s\right) $ is a
globally Lipschitz function, thanks to $L_{\zeta }^{1}$-contraction result
(consequence of the $T$-accretivity of $A$ in $X=L_{\zeta }^{1}\left( \Omega
\right) $ (see also \cite{Aronson1982} or \cite{Benilan-Crandall-Sacks})),
we get 
\begin{equation*}
v\left( t,x\right) \leq u_{\varepsilon }\left( t,x\right) ,\mbox{ a.e. in }%
\left( 0,\infty \right) \times \Omega .
\end{equation*}%
Passing to th elimit as $\varepsilon \downarrow 0$ we obtain the wanted
inequality.

\noindent \textbf{Step 4: Treatment of unbounded nonnegative initial data $%
u_{0}$.} Let $u_{0}\in L_{\delta }^{1}\left( \Omega \right) $, $u_{0}\geq 0$
and let 
\begin{equation*}
{u_{0,n}}\left( x\right) =\inf \left\{ {{u_{0}}\left( x\right) ,n}\right\} .
\end{equation*}%
Then $u_{0,n}\in L^{\infty }\left( \Omega \right) $, $u_{0,n}\geq 0$ and ${%
u_{0,n}}\uparrow {u_{0}}$ in $L_{\delta }^{1}\left( \Omega \right) $ as $%
n\uparrow +\infty $. Then, as before we can apply the comparison principle
to deduce that, for any $\epsilon >0$, if ${u_{\epsilon ,n}}$ is the
(unique) solution of problem $(P_{\epsilon })$, then

\begin{equation*}
{u_{\epsilon ,n_{1}}\leq u_{\epsilon ,n_{2}}}\text{ in }\left( {0,\infty }%
\right) \times \Omega \text{, if }n_{1}\leq n_{2}.
\end{equation*}%
Moreover, we have the uniform bound 
\begin{equation}
0\leq u_{n}\left( t,x\right) \leq U\left( t,x\right) ,\mbox{ a.e. in }\left(
0,T\right) \times \Omega ,  \label{Inequ. Porous media subsolution}
\end{equation}%
with $U\in \mathcal{C}\left( \left[ 0,T\right] ;L_{\zeta }^{1}\left( \Omega
\right) \right) $ the unique $L_{\zeta }^{1}$-mild solution of the
homogeneous problem 
\begin{equation}
\left\{ 
\begin{split}
& {\partial _{t}}U-\Delta {U^{m}}=0,\mbox{ in }\left( {0,T}\right) \times
\Omega , \\
& U=0,\mbox{ on }\left( {0,T}\right) \times \partial \Omega , \\
& U\left( {0,x}\right) ={u_{0}}\left( x\right) ,\mbox{ in }\Omega .
\end{split}%
\right.  \label{3.11}
\end{equation}%
Indeed, it suffices to use that for any $n$ and $\epsilon >0$ we have $-{{%
g_{\varepsilon }}\left( {u_{\epsilon ,n}}\right) }\left( t,x\right) \leq 0$
in $\left( 0,T\right) \times \Omega $, and to use the comparison principle
for the unperturbed nonlinear diffusion problem. Then, passing to the limit,
as in Step 2, we deduce that if ${u_{n}}$ is the \textit{maximal} $L_{\zeta
}^{1}$-mild solution of (\ref{P}) associated to $u_{0,n}\in L^{\infty
}\left( \Omega \right) $ then 
\begin{equation*}
{u_{n_{1}}\leq u_{n_{2}}}\text{ in }\mathcal{C}\left( \left[ 0,T\right]
;L_{\zeta }^{1}\left( \Omega \right) \right) \text{, if }n_{1}\leq n_{2}.
\end{equation*}%
Moreover, 
\begin{equation*}
{u_{n_{1}}^{-\beta }\geq u_{n_{2}}^{-\beta }}\text{ on }\left\{ (t,x)\in
\left( {0,\infty }\right) \times \Omega ,{u_{n_{1}}(t,x)>0}\right\} \text{,
if }n_{1}\leq n_{2},
\end{equation*}%
and that, in fact, $\left\{ {u_{n_{1}}>0}\right\} \supset \left\{ {%
u_{n_{2}}>0}\right\} $. Then $\Upsilon _{n}:=-{u_{n}^{-\beta }}{\chi
_{\left\{ {u}_{n}{>0}\right\} }}$, is a monotone sequence of nonnegative
functions in $L^{1}\left( 0,T;L_{\delta }^{1}\left( \Omega \right) \right) $
which converges to some $\Upsilon $ in $L^{1}\left( 0,T;L_{\delta
}^{1}\left( \Omega \right) \right) $ and thus we can apply, again the
extension of the Benilan-Crandall-Saks \cite{Benilan-Crandall-Sacks}
argument to pass to the limit of $L_{\zeta }^{1}-$mild solutions of problems
of the type (\ref{3.4}) and thus we get that $u_{n}\rightarrow u$ in $%
\mathcal{C}\left( \left[ 0,T\right] ;L_{\zeta }^{1}\left( \Omega \right)
\right) $ with $u$ the unique $L_{\zeta }^{1}\left( \Omega \right) $-mild
solution of the problem \eqref{3.5}, as $n\uparrow +\infty $. Arguing as in
Step 2 we get that $\Upsilon =-{u^{-\beta }}{\chi _{\left\{ {u>0}\right\} }}$
and thus ${u^{-\beta }}{\chi _{\left\{ {u>0}\right\} }}\in $ $L^{1}\left(
0,T;L_{\delta }^{1}\left( \Omega \right) \right) $. The proof of the
maximality is again similar to the arguments of Step 4.

\noindent \textbf{Step 5: Gradient estimate for $u_{0}\in $}$L_{\zeta
}^{1}\left( \Omega \right) $\textbf{.}

\noindent Notice that, from (\ref{Inequ. Porous media subsolution}) we get
(after passing to the limit, as $n\uparrow +\infty $) 
\begin{equation}
0\leq u\left( t,x\right) \leq U\left( t,x\right) ,\mbox{ a.e. in }\left(
0,T\right) \times \Omega ,
\end{equation}%
On the other hand, by applying the smoothing effects shown in Veron \cite%
{Veron1979} (see also \cite{Rakotoson2011} for the semilinear case), and the
explicit sharp estimate given in \cite[(17.32)]{Vazquez2007} (see a
different proof via other rearrangement arguments in \cite{Diaz1991}
combined with Theorem 3.1 of \cite{Diaz2010}), we know that for any $m\geq 1$
\begin{equation}
{\left\Vert {U\left( t\right) }\right\Vert _{{L^{\infty }}\left( \Omega
\right) }}\leq \frac{{C\left( \Omega \right) }}{{t^{\alpha }}}\left\Vert {%
u_{0}}\right\Vert _{L_{\zeta }^{1}\left( \Omega \right) }^{\sigma },
\label{Estimate L^infinity}
\end{equation}%
with%
\begin{equation*}
\text{ }\alpha =\frac{N}{N(m-1)+2}\text{ and }\sigma =\frac{2}{N(m-1)+2}.
\end{equation*}%
In the special case of $m>1$ we have an universal estimate for $U$ (see,
e.g. Proposition 5.17 of \cite{Vazquez2007}) 
\begin{equation}
{\left\Vert {U\left( t\right) }\right\Vert _{{L^{\infty }}\left( \Omega
\right) }\leq C(m,N)R}^{\frac{2}{m-1}}t^{-\frac{1}{m-1}}
\label{Universal bound pM}
\end{equation}%
where $R$ is the radius of a ball containing $\Omega $.

\noindent Thus the same estimates \eqref{Estimate L^infinity}, for $m\geq 1,$
and \eqref{Universal bound pM}, for $m>1$, also hold for $u$. Using Lemma %
\ref{lemma2} we get that for any $t>0$, a.e. $x\in \Omega ,$ and for any $%
\lambda \in (0,t)$ we have 
\begin{equation*}
{\left\vert {\nabla u_{\varepsilon }^{1/\gamma }\left( t{,x}\right) }%
\right\vert ^{2}}\leq C\left( \frac{{\left\Vert {u(t-\lambda )}\right\Vert _{%
{L^{\infty }}\left( \Omega \right) }^{1+\beta }}}{t-\lambda }{+1}\right)
\leq C\left( \frac{C(\Omega )^{1+\beta }{\left\Vert {u}_{0}\right\Vert _{{%
L^{1}}\left( \Omega \right) }^{(1+\beta )\sigma }}}{(t-\lambda )^{\alpha +1}}%
{+1}\right) ,
\end{equation*}%
if $m\geq 1$, or 
\begin{equation*}
{\left\vert {\nabla u_{\varepsilon }^{1/\gamma }\left( t{,x}\right) }%
\right\vert ^{2}}\leq C\left( \frac{{\left\Vert {u(t-\lambda )}\right\Vert _{%
{L^{\infty }}\left( \Omega \right) }^{1+\beta }}}{t-\lambda }{+1}\right)
\leq C\left( \frac{\left[ {C(m,N)R}^{\frac{2}{m-1}}(t-\lambda )^{-\frac{1}{%
m-1}}\right] ^{1+\beta }}{(t-\lambda )}{+1}\right) \text{,}
\end{equation*}%
if $m>1$. Passing to the limit, first as $\lambda \downarrow 0$ and then as $%
\varepsilon \downarrow 0,$ (using the convergence of the Step 2 and weak-$%
\star $ convergence in ${L^{\infty }}\left( \Omega \right) $) we get the
pointwise gradient estimate given in ii) of Theorem \ref{theorem1}, with $%
\omega =\alpha +1$ if $m\geq 1$ and $\omega =(\beta +m)/(m-1)$ if $m>1$.

\noindent Now, the proof of the fact that the \emph{maximal} $L^{1}$-mild
solution is H\"{o}lder continuous on $({0,T]}\times \overline{\Omega }$ is a
simple consequence of Proposition \ref{proposition1} and the above
convergence arguments.

\noindent \textbf{Step 6: Case $m+\beta <2$: gradient convergence and proof
of iii) of Theorem \ref{theorem1}.}

\noindent In order to prove part iii) of Theorem \ref{theorem1} we shall use
other type of convergence arguments. As a matter of fact, we will prove a
stronger result showing the gradient convergence as $\varepsilon \downarrow
0 $: 
\begin{equation*}
\nabla {u_{\varepsilon }}\rightarrow \nabla u,\mbox{ a.e. in }\left( {0,T}%
\right) \times \Omega ,
\end{equation*}%
up to a subsequence. Indeed, from the equations satisfied by $u_{\varepsilon
}$ and $u_{\varepsilon ^{\prime }}$ for any $\varepsilon >\varepsilon
^{\prime }>0$, we have 
\begin{equation*}
{\partial _{t}}\left( {{u_{\varepsilon }}-{u_{\varepsilon ^{\prime }}}}%
\right) -\left( {\Delta u_{\varepsilon }^{m}-\Delta u_{\varepsilon ^{\prime
}}^{m}}\right) +{g_{\varepsilon }}\left( {u_{\varepsilon }}\right) -{%
g_{\varepsilon ^{\prime }}}\left( {{u_{\varepsilon ^{\prime }}}}\right) =0.
\end{equation*}%
For any $\delta >0$, let us define 
\begin{equation*}
{T_{\delta }}\left( s\right) =\left\{ 
\begin{split}
& s,\mbox{ if }\left\vert s\right\vert <\delta , \\
& \delta \,\mathrm{sign}\left( s\right) ,\mbox{ if }\left\vert s\right\vert
\geq \delta ,
\end{split}%
\right. \mbox{ and }{S_{\delta }}\left( r\right) =\int_{0}^{r}{{T_{\delta }}%
\left( s\right) ds}.
\end{equation*}%
For any $0<\tau <T<\infty $, by using $T_{\delta }\left( u_{\varepsilon
}-u_{\varepsilon ^{\prime }}\right) $ as a test function in \eqref{3.6}, and
integrating both sides of \eqref{3.6} on $\left( \tau ,T\right) \times
\Omega $, we obtain 
\begin{align*}
& \int_{\Omega }{{S_{\delta }}\left( {{u_{\varepsilon }}-{u_{\varepsilon
^{\prime }}}}\right) \left( {T,x}\right) dx}+\int_{\tau }^{T}{\int_{\Omega }{%
\left( {mu_{\varepsilon }^{m-1}\nabla {u_{\varepsilon }}-mu_{\varepsilon
^{\prime }}^{m-1}\nabla {u_{\varepsilon ^{\prime }}}}\right) \cdot \nabla {%
T_{\delta }}\left( {{u_{\varepsilon }}-{u_{\varepsilon ^{\prime }}}}\right)
dxdt}} \\
& \hspace{1cm}+\int_{\tau }^{T}{\int_{\Omega }{\left( {{g_{\varepsilon }}%
\left( {u_{\varepsilon }}\right) -{g_{\varepsilon ^{\prime }}}\left( {{%
u_{\varepsilon ^{\prime }}}}\right) }\right) {T_{\delta }}\left( {{%
u_{\varepsilon }}-{u_{\varepsilon ^{\prime }}}}\right) dxdt}}=\int_{\Omega }{%
{S_{\delta }}\left( {{u_{\varepsilon }}-{u_{\varepsilon ^{\prime }}}}\right)
\left( {\tau ,x}\right) dx}.
\end{align*}%
It follows from the facts $S_{\delta }\left( r\right) \geq 0$ and $S_{\delta
}\left( r\right) \leq \delta \left\vert r\right\vert $, $\forall r\in 
\mathbb{R}$ that 
\begin{align*}
& \int_{\tau }^{T}{\int_{\Omega }{mu_{\varepsilon }^{m-1}\nabla \left( {{%
u_{\varepsilon }}-{u_{\varepsilon ^{\prime }}}}\right) \cdot \nabla {%
T_{\delta }}\left( {{u_{\varepsilon }}-{u_{\varepsilon ^{\prime }}}}\right)
dxdt}} \\
& \hspace{1cm}+\int_{\tau }^{T}{\int_{\Omega }{m\left( {u_{\varepsilon
}^{m-1}-u_{\varepsilon ^{\prime }}^{m-1}}\right) \nabla {u_{\varepsilon
^{\prime }}}\cdot \nabla {T_{\delta }}\left( {{u_{\varepsilon }}-{%
u_{\varepsilon ^{\prime }}}}\right) dxdt}} \\
& \hspace{1cm}+\int_{\tau }^{T}{\int_{\Omega }{\left( {{g_{\varepsilon }}%
\left( {u_{\varepsilon }}\right) -{g_{\varepsilon ^{\prime }}}\left( {{%
u_{\varepsilon ^{\prime }}}}\right) }\right) {T_{\delta }}\left( {{%
u_{\varepsilon }}-{u_{\varepsilon ^{\prime }}}}\right) dxdt}}\leq \delta
\int_{\Omega }{\left\vert {\left( {{u_{\varepsilon }}-{u_{\varepsilon
^{\prime }}}}\right) \left( {\tau ,x}\right) }\right\vert dx}.
\end{align*}%
Since $\left\vert T_{\delta }\left( s\right) \right\vert \leq \delta $, $%
\forall s\in \mathbb{R}$, we obtain from the last inequality 
\begin{align}
& \iint_{\left\{ {\left\vert {{u_{\varepsilon }}-{u_{\varepsilon ^{\prime }}}%
}\right\vert <\delta }\right\} }{u_{\varepsilon }^{m-1}{{\left\vert {\nabla
\left( {{u_{\varepsilon }}-{u_{\varepsilon ^{\prime }}}}\right) }\right\vert 
}^{2}}dxdt}\leq 4\delta {\left\Vert {u_{0}}\right\Vert _{{L^{1}}\left(
\Omega \right) }}  \notag \\
& \hspace{1cm}+\int_{\tau }^{T}{\int_{\Omega }{\left\vert {\left( {%
u_{\varepsilon }^{m-1}-u_{\varepsilon ^{\prime }}^{m-1}}\right) \nabla {%
u_{\varepsilon ^{\prime }}}\cdot \nabla {T_{\delta }}\left( {{u_{\varepsilon
}}-{u_{\varepsilon ^{\prime }}}}\right) }\right\vert dxdt}}.  \label{3.7}
\end{align}%
Then, from \eqref{GE} and the Dominated Convergence Theorem we get 
\begin{equation*}
\int_{\tau }^{T}{\int_{\Omega }{\left\vert {\left( {u_{\varepsilon
}^{m-1}-u_{\varepsilon ^{\prime }}^{m-1}}\right) \nabla {u_{\varepsilon
^{\prime }}}\cdot \nabla {T_{\delta }}\left( {{u_{\varepsilon }}-{%
u_{\varepsilon ^{\prime }}}}\right) }\right\vert dxdt}}\rightarrow 0,%
\mbox{
as }\varepsilon ,\varepsilon ^{\prime }\downarrow 0,
\end{equation*}%
and 
\begin{equation*}
\iint_{\left\{ {\left\vert {{u_{\varepsilon }}-{u_{\varepsilon ^{\prime }}}}%
\right\vert <\delta }\right\} }{u_{\varepsilon }^{m-1}{{\left\vert {\nabla
\left( {{u_{\varepsilon }}-{u_{\varepsilon ^{\prime }}}}\right) }\right\vert 
}^{2}}dxdt}\leq 4\delta {\left\Vert {u_{0}}\right\Vert _{{L^{1}}\left(
\Omega \right) }}+o\left( {\varepsilon ,\varepsilon ^{\prime }}\right) ,
\end{equation*}%
where $o\left( \varepsilon ,\varepsilon ^{\prime }\right) \rightarrow 0$ as $%
\varepsilon ,\varepsilon ^{\prime }\downarrow 0$. Moreover, it is clear that 
\begin{equation*}
\iint_{\left\{ {{u_{\varepsilon }}>\delta ,\,\left\vert {{u_{\varepsilon }}-{%
u_{\varepsilon ^{\prime }}}}\right\vert <\delta }\right\} }{{{\left\vert {%
\nabla \left( {{u_{\varepsilon }}-{u_{\varepsilon ^{\prime }}}}\right) }%
\right\vert }^{2}}dxdt}\leq {\delta ^{1-m}}\iint_{\left\{ {{u_{\varepsilon }}%
>\delta ,\,\left\vert {{u_{\varepsilon }}-{u_{\varepsilon ^{\prime }}}}%
\right\vert <\delta }\right\} }{u_{\varepsilon }^{m-1}{{\left\vert {\nabla
\left( {{u_{\varepsilon }}-{u_{\varepsilon ^{\prime }}}}\right) }\right\vert 
}^{2}}dxdt}.
\end{equation*}%
It follows from the last inequality that 
\begin{equation*}
\iint_{\left\{ {{u_{\varepsilon }}>\delta ,\,\left\vert {{u_{\varepsilon }}-{%
u_{\varepsilon ^{\prime }}}}\right\vert <\delta }\right\} }{{{\left\vert {%
\nabla \left( {{u_{\varepsilon }}-{u_{\varepsilon ^{\prime }}}}\right) }%
\right\vert }^{2}}dxdt}\leq 4{\delta ^{2-m}}{\left\Vert {u_{0}}\right\Vert _{%
{L^{1}}\left( \Omega \right) }}+{\delta ^{1-m}}o\left( {\varepsilon
,\varepsilon ^{\prime }}\right) .
\end{equation*}%
Thanks to \eqref{GE}, we obtain 
\begin{equation*}
\iint_{\left\{ {{u_{\varepsilon }}\leq \delta ,\,\left\vert {{u_{\varepsilon
}}-{u_{\varepsilon ^{\prime }}}}\right\vert <\delta }\right\} }{{{\left\vert 
{\nabla {u_{\varepsilon }}}\right\vert }^{2}}dxdt}\leq C\iint_{\left\{ {{%
u_{\varepsilon }}\leq \delta ,\,\left\vert {{u_{\varepsilon }}-{%
u_{\varepsilon ^{\prime }}}}\right\vert <\delta }\right\} }{u_{\varepsilon
}^{2\left( {1-\frac{1}{\gamma }}\right) }dxdt}\leq CT\left\vert \Omega
\right\vert {\delta ^{2\left( {1-\frac{1}{\gamma }}\right) }},
\end{equation*}%
where the constant $C>0$ is independent of $\varepsilon $, $\delta $. Since $%
u_{\varepsilon }\geq u_{\varepsilon ^{\prime }}$, and by the same argument,
we also obtain 
\begin{equation*}
\iint_{\left\{ {{u_{\varepsilon }}\leq \delta ,\,\left\vert {{u_{\varepsilon
}}-{u_{\varepsilon ^{\prime }}}}\right\vert <\delta }\right\} }{{{\left\vert 
{\nabla {u_{\varepsilon ^{\prime }}}}\right\vert }^{2}}dxdt}\leq C{\delta
^{2\left( {1-\frac{1}{\gamma }}\right) }}.
\end{equation*}%
Combining these, we get 
\begin{equation*}
\iint_{\left\{ {\left\vert {{u_{\varepsilon }}-{u_{\varepsilon ^{\prime }}}}%
\right\vert <\delta }\right\} }{{{\left\vert {\nabla \left( {{u_{\varepsilon
}}-{u_{\varepsilon ^{\prime }}}}\right) }\right\vert }^{2}}dxdt}\lesssim {%
\delta ^{2-m}}{\left\Vert {u_{0}}\right\Vert _{{L^{1}}\left( \Omega \right) }%
}+{\delta ^{1-m}}o\left( {\varepsilon ,\varepsilon ^{\prime }}\right) +{%
\delta ^{2\left( {1-\frac{1}{\gamma }}\right) }}.
\end{equation*}%
Here we used the notation $A\lesssim B$ in the sense that there is a
constant $c>0$ such that $A\leq cB$. Thanks to \eqref{GE}, and the fact that 
$u_{\varepsilon }\rightarrow u$, we obtain 
\begin{equation*}
\iint_{\left\{ {\left\vert {{u_{\varepsilon }}-{u_{\varepsilon ^{\prime }}}}%
\right\vert \geq \delta }\right\} }{{{\left\vert {\nabla \left( {{%
u_{\varepsilon }}-{u_{\varepsilon ^{\prime }}}}\right) }\right\vert }^{2}}%
dxdt}\leq Cmeas\left( {\left\{ {\left\vert {{u_{\varepsilon }}-{%
u_{\varepsilon ^{\prime }}}}\right\vert \geq \delta }\right\} }\right) \leq
Co\left( {\varepsilon ,\varepsilon ^{\prime }}\right) ,
\end{equation*}%
with $C=C\left( {m,\beta ,N,\tau ,T,{{\left\Vert {u_{0}}\right\Vert }%
_{\infty }}}\right) $. It follows from that 
\begin{equation*}
\int_{\tau }^{T}{\int_{\Omega }{{{\left\vert {\nabla \left( {{u_{\varepsilon
}}-{u_{\varepsilon ^{\prime }}}}\right) }\right\vert }^{2}}dxdt}}\lesssim {%
\delta ^{2-m}}{\left\Vert {u_{0}}\right\Vert _{{L^{1}}\left( \Omega \right) }%
}+\left( {1+{\delta ^{1-m}}}\right) o\left( {\varepsilon ,\varepsilon
^{\prime }}\right) +{\delta ^{2\left( {1-\frac{1}{\gamma }}\right) }}.
\end{equation*}%
Hence, 
\begin{equation*}
\mathop {\limsup }\limits_{\varepsilon \downarrow 0}\int_{\tau }^{T}{%
\int_{\Omega }{{{\left\vert {\nabla \left( {{u_{\varepsilon }}-{%
u_{\varepsilon ^{\prime }}}}\right) }\right\vert }^{2}}dxdt}}\leq {\delta
^{2-m}}{\left\Vert {u_{0}}\right\Vert _{{L^{1}}\left( \Omega \right) }}+{%
\delta ^{2\left( {1-\frac{1}{\gamma }}\right) }}.
\end{equation*}%
The last inequality holds for any $\delta >0$ and since, now, $m+\beta <2$,
we obtain 
\begin{equation*}
\mathop {\limsup }\limits_{\varepsilon \downarrow 0}\int_{\tau }^{T}{%
\int_{\Omega }{{{\left\vert {\nabla \left( {{u_{\varepsilon }}-{%
u_{\varepsilon ^{\prime }}}}\right) }\right\vert }^{2}}dxdt}}=0.
\end{equation*}%
Consequently, we have 
\begin{equation*}
\nabla u_{\varepsilon }\rightarrow \nabla u,\mbox{ in }L^{2}\left( \left(
\tau ,T\right) \times \Omega \right) .
\end{equation*}%
Up to a subsequence, we deduce $\nabla u_{\varepsilon }\rightarrow \nabla u$
a.e. in $\left( \tau ,T\right) \times \Omega $. A diagonal argument implies
that there is a subsequence of $\left( u_{\varepsilon }\right) _{\varepsilon
>0}$ (still denoted as $\left( u_{\varepsilon }\right) _{\varepsilon >0}$)
such that 
\begin{equation*}
\nabla u_{\varepsilon }\rightarrow \nabla u,\mbox{ a.e. in }\left( 0,\infty
\right) \times \Omega .
\end{equation*}%
Hence, $u$ also satisfies the gradient estimates \eqref{GE} and \eqref{GE1}.

\noindent This puts an end to the proof of Theorem \ref{theorem1}. \hfill $%
\square$

\begin{remark}
An alternative proof of the regularity $u\in \mathcal{C}\left( \left[
0,\infty \right) ;L^{1}\left( \Omega \right) \right) $\textbf{,} in part
iii) of Theorem \ref{theorem1}, when $u_{0}\in {L^{\infty }}\left( \Omega
\right) ,$ is the following: for any $1<p<2$, thanks to Lemma \ref{lemma2},
we have that for any finite time $T>0$ 
\begin{equation}
\int_{0}^{T}{\int_{\Omega }{{{\left\vert {\nabla u}\right\vert }^{p}}dxdt}}%
\leq C\int_{0}^{T}{\int_{\Omega }{{u^{p\left( {1-\frac{1}{\gamma }}\right) }}%
{{\left( {{t^{-1}}\left\Vert {u_{0}}\right\Vert _{{L^{\infty }}\left( \Omega
\right) }^{1+\beta }+1}\right) }^{p/2}}dxdt}}\leq {C_{1}},  \label{3.21}
\end{equation}%
where $C_{1}>0$ only depends on $T$, $\Omega $, ${\left\Vert {u_{0}}%
\right\Vert _{{L^{\infty }}\left( \Omega \right) }}$, and the parameters
involved. Since $u$ is bounded on $\left( 0,\infty \right) \times \Omega $,
it follows from \eqref{3.21} that 
\begin{equation*}
\nabla {u^{m}}\in {L^{p}}\left( {\left( {0,T}\right) ,W_{0}^{1,p}\left(
\Omega \right) }\right) .
\end{equation*}%
This implies that 
\begin{equation*}
{\partial _{t}}u=\mathrm{div}\left( {\nabla {u^{m}}}\right) -{u^{-\beta }}{%
\chi _{\left\{ {u>0}\right\} }}\in {L^{p}}\left( {\left( {0,T}\right)
,W_{0}^{-1,p}\left( \Omega \right) }\right) \cap {L^{1}}\left( {\left( {0,T}%
\right) \times \Omega }\right) ,
\end{equation*}%
where $W^{-1,p}\left( \Omega \right) $ is the dual space of $%
W_{0}^{1,p}\left( \Omega \right) $. Then, by a compactness embedding (see 
\cite{Porreta1999}), we obtain $u\in \mathcal{C}\left( \left[ 0,T\right]
,L^{1}\left( \Omega \right) \right) $.
\end{remark}

The rest of this section is devoted to consider the associated Cauchy
problem for initial data $u_{0}\in L^{1}\left( \mathbb{R}^{N}\right) \cap
L^{\infty }\left( \mathbb{R}^{N}\right) $. The existence of solutions to the
Cauchy problem \eqref{PC} can be obtained as a consequence of Theorem \ref%
{theorem1}. Here is a simplified statement:

\begin{theorem}
\label{theorem3} Assume $m$, $N$, $\beta $ as in Theorem \ref{theorem1}. Let 
$u_{0}\in L^{1}\left( \mathbb{R}^{N}\right) \cap L^{\infty }\left( \mathbb{R}%
^{N}\right) $, $u_{0}\geq 0$. Then, problem \eqref{PC} has a weak solution $%
u\in \mathcal{C}\left( {\left[ {0,\infty }\right) ,{L^{1}}\left( {{\mathbb{R}%
^{N}}}\right) }\right) \cap {L^{\infty }}\left( {\left( {0,\infty }\right)
\times {\mathbb{R}^{N}}}\right) $ satisfying \eqref{PC} in the sense of
distributions: 
\begin{equation*}
\int_{0}^{\infty }{\int_{{\mathbb{R}^{N}}}{\left( {\ -u{\varphi _{t}}-{u^{m}}%
\Delta \varphi +{u^{-\beta }}{\chi _{\left\{ {u>0}\right\} }}\varphi }%
\right) dxdt}}=0,\ \forall \varphi \in \mathcal{D}\left( {\left( {0,\infty }%
\right) \times {\mathbb{R}^{N}}}\right) .
\end{equation*}%
Moreover, the gradient estimates of Lemma \ref{lemma2} remain valid with $%
C=C\left( m,\beta ,N,{\left\Vert {u}_{0}\right\Vert _{{L^{1}\left( \Omega
\right) }}}\right) $ for any $m\geq 1$.
\end{theorem}

\begin{proof}
We will start by constructing a sequence $\left( u_{\varepsilon }\right)
_{\varepsilon >0}$ of solutions of the regularized problem 
\begin{equation}
\left\{ 
\begin{split}
& {\partial _{t}}u-\Delta {u^{m}}+{g_{\varepsilon }}\left( u\right) =0,%
\mbox{ in }\left( {0,\infty }\right) \times {\mathbb{R}^{N}}, \\
& u\left( {0,x}\right) ={u_{0}}\left( x\right) ,\mbox{ in }{\mathbb{R}^{N}}.
\end{split}%
\right.  \label{5.1}
\end{equation}%
After that we will prove that $u_{\varepsilon }\rightarrow u$, with $u$ a
weak solution of problem \eqref{PC}.

\noindent The proof of the construction of $\left( u_{\varepsilon }\right)
_{\varepsilon >0}$ is quite similar to the one given in the proof of Theorem %
\ref{theorem1}. Thus, we just sketch out the main idea. We start by
considering the approximate problem over $\left( 0,\infty \right) \times
B_{R}$, for any $R>0$, taking as initial data the function ${u_{0}}{\chi _{{%
B_{R}}}}$. By some classical results on the accretive operators theory (see,
e.g., \cite{Aronson1982,Vazquez2007}) we know that there is a unique weak
solution $u_{\varepsilon ,R}$ of the approximate problem in $\left( 0,\infty
\right) \times B_{R}$. and that (from the construction of the initial datum
on $B_{R}$), for any $\varepsilon ,\,R>0$, we have the estimates 
\begin{equation*}
{\left\Vert {{u_{\varepsilon ,R}}\left( t\right) }\right\Vert _{{L^{1}}%
\left( {{B_{R}}}\right) }}\leq {\left\Vert {{u_{0}}}\right\Vert _{{L^{1}}%
\left( {{\mathbb{R}^{N}}}\right) }},\ \forall t>0,
\end{equation*}%
and 
\begin{equation*}
{\left\Vert {{u_{\varepsilon ,R}}\left( t\right) }\right\Vert _{{L^{\infty }}%
\left( {{B_{R}}}\right) }}\leq {\left\Vert {{u_{0}}}\right\Vert _{{L^{\infty
}}\left( {{\mathbb{R}^{N}}}\right) }},\ \forall t>0.
\end{equation*}%
Thanks to Lemma \ref{lemma2}, we also know that 
\begin{equation*}
{\left\vert {\nabla u_{\varepsilon ,R}^{\frac{1}{\gamma }}\left( {t,x}%
\right) }\right\vert ^{2}}\leq C\left( {{t^{-1}}\left\Vert {{u_{0}}}%
\right\Vert _{{L^{\infty }}\left( {{\mathbb{R}^{N}}}\right) }^{1+\beta }+1}%
\right) ,\mbox{ in }\left( {0,\infty }\right) \times {B_{R}}.
\end{equation*}%
Moreover, for any fixed $\varepsilon >0$, it follows from the $L^{1}$%
-contraction property (for the unperturbed nonlinear diffusion problem) that
the sequence $\left( u_{\varepsilon ,R}\right) _{R>0}$ is pointwise
non-decreasing. Thus, there exists a function, denoted by $u_{\varepsilon }$%
, such that $u_{\varepsilon ,R}\uparrow u_{\varepsilon }$ as $R\rightarrow
\infty $. Consequently, $u_{\varepsilon }$ satisfies the corresponding
estimates for the respective $L^{1}\left( \mathbb{R}^{N}\right) $ and $%
L^{\infty }\left( \mathbb{R}^{N}\right) $ norms. Moreover, since $%
g_{\varepsilon }\left( \cdot \right) $ is a globally Lipschitz function,
the classical regularity result (see, e.g., \cite{Aronson1982,Vazquez2007})
implies that 
\begin{equation*}
\nabla u_{\varepsilon ,R}^{m}\rightarrow \nabla u_{\varepsilon } ^{m},%
\mbox{
a.e. in }\left( 0,\infty \right) \times \mathbb{R}^{N},
\end{equation*}%
up to a subsequence. Similarly as in the proof of Theorem \ref{theorem1}, we
observe that $\left( u_{\varepsilon }\right) _{\varepsilon >0}$ is a
non-decreasing sequence. Thus, there exists a function $u$ such that $%
u_{\varepsilon }\downarrow u$ in $\left( 0,\infty \right) \times \mathbb{R}%
^{N}$, as $\varepsilon \downarrow 0$. Then, we mimic the different steps in
the proof of Theorem \ref{theorem1} to pass to the limit as $\varepsilon
\downarrow 0$. We point out that the continuous dependence in $\mathcal{C}%
\left( \left[ 0,T\right] ,L^{1}\left( \mathbb{R}^{N}\right) \right) $ is
quite similar to the case of a bounded domain $\Omega $ since we do not need
to approximate the nonlinear term $\psi (u)=u^{m}$. Then we get that $u$ is
a weak solution of equation \eqref{PC} and in fact $u$ is the maximal
solution of problem \eqref{PC}. 
\end{proof}

\begin{remark}
In a similar way to the case of bounded domains, the accretivity in ${L^{1}}%
\left( {{\mathbb{R}^{N}}}\right) $ can be replaced by the accretivity in
some weighted spaces ${L_{\rho _{\alpha }}^{1}}\left( {{\mathbb{R}^{N}}}%
\right) $ allowing to get the existence of solutions for the Cauchy problem
for a more general class of initial data ${u_{0}}\left( x\right) $ growing
with $\left\vert x\right\vert $, as $\left\vert x\right\vert \rightarrow
+\infty $. That was started with the paper \cite{Benilan-Crandall} and then
developed and improved by several authors (see the exposition made in
Chapter 12 of \cite{Vazquez2007}). The mentioned accretivity in ${L_{\rho
_{\alpha }}^{1}}\left( {{\mathbb{R}^{N}}}\right) $ holds, for any, $m>0$ and 
$N\geq 3$, for the weight given by 
\begin{equation*}
\rho _{\alpha }(x)=\frac{1}{(1+\left\vert x\right\vert ^{2})^{\alpha }}
\end{equation*}
with $\alpha $ given such that $0<\alpha \leq (N-2)/2.$ For other values of $%
N$ and $\alpha >0$ there is only existence of local in time solutions of the
Cauchy Problem (\cite{Vazquez2007}). This property could be used to get some
generalizations of the results of \cite{Guo 2007} for the study of \eqref{PC}
when $m>1$, but we will not pursuit this goal in this paper.
\end{remark}

\section{Qualitative properties}

We start by recalling that the existence of a $L_{\delta }^{1}$\emph{-mild
solution} of \eqref{P(1)} (for more regular solutions see, e.g. Subsection
5.5.1 of \cite{Vazquez2007}).

\begin{definition}
\label{definition3 copy(1)} Let $u_{0}\in L_{\delta }^{1}\left( \Omega
\right) $, $u_{0}\geq 0$, and $T>0$. A nonnegative function $u\in \mathcal{C}%
\left( \left[ 0,T\right] ;L_{\delta }^{1}\left( \Omega \right) \right) $ is
called a $L_{\delta }^{1}$\emph{-mild solution} of \eqref{P(1)} if ${%
u^{-\beta }}{\chi _{\left\{ {u>0}\right\} }}\in {L^{1}}\left( {{0,T:}%
L_{\delta }^{1}\left( \Omega \right) }\right) $ coincides with the unique $%
L_{\delta }^{1}$\emph{-}mild solution of the problem 
\begin{equation}
\left\{ 
\begin{split}
{\partial _{t}}u-\Delta {u^{m}}=f,& \mbox{ in }\left( {0,T}\right) \times
\Omega , \\
u^{m}=1,& \mbox{ on }\left( {0,T}\right) \times \partial \Omega , \\
u\left( {0,x}\right) ={u_{0}}\left( x\right) ,& \mbox{ in }\Omega ,
\end{split}%
\right.  \label{eq nonhomogeneous P(1)}
\end{equation}%
\eqref{eq nonhomogeneous} where $f:=-{u^{-\beta }\chi _{\left\{ {u>0}%
\right\} }}$.
\end{definition}

The existence and uniqueness of a $L_{\delta }^{1}$\emph{-mild solution} of %
\eqref{eq nonhomogeneous P(1)} for a given $f\in {L^{1}}\left( {{0,T:}%
L_{\delta }^{1}\left( \Omega \right) }\right) $ is an easy modification of
the results of \cite{Brezis 1971}, \cite{Veron2004}, Theorem 1.10 of \cite%
{Diaz-Mingazzini} and Step 2 of the above Section. Indeed, given $f\in
L_{\delta }^{1}\left( \Omega \right) $ and $\lambda \geq 0$, we start by
recalling the definition of very weak solution of the stationary problem

\begin{equation}
P(f,\lambda ,1)=\left\{ 
\begin{array}{cc}
-\Delta (\left\vert {u}\right\vert {^{m-1}u)+\lambda u}={f} & \text{in }%
\Omega , \\ 
\left\vert {u}\right\vert {^{m-1}u}=1 & \text{on }\partial \Omega .%
\end{array}%
\right.
\end{equation}

\begin{definition}
\label{Definition 4 copy(1)} Given $f\in L_{\delta }^{1}\left( \Omega
\right) $ and $\lambda \geq 0$, a function $u\in $ $L_{\delta }^{1}\left(
\Omega \right) $ is called a very weak solution of $P(f,\lambda )$ if ${%
\left\vert u\right\vert ^{m-1}}u\in {L^{1}}\left( \Omega \right) $ and for
any $\psi {\in W}^{2,\infty }(\Omega )\cap {W}_{0}^{1,\infty }(\Omega )$ $\
\int_{\Omega }{u\left( x\right) }^{m}\Delta \psi {(x)dx}+{\int_{\Omega
}\lambda {{u(x)}\psi {(x)}dx}}={\int_{\Omega }}f(x)\psi {(x){dx-}%
\int_{\partial \Omega }}\frac{\partial \psi }{\partial n}{(x){dx.}}$
\end{definition}

In a completely similar way to Step 2 of the above Section we have

\begin{lemma}
\label{Lemma P(1)} Let $X=L_{\zeta }^{1}\left( \Omega \right) $, $m>0$ and
define the operator $A:D(A)\rightarrow X$ given by 
\begin{equation*}
\begin{array}{cc}
Au=-\Delta (\left\vert {u}\right\vert {^{m}u):=f} & u\in D(A)%
\end{array}%
\end{equation*}%
with%
\begin{equation*}
D(A)=\{u\in L_{\zeta }^{1}\left( \Omega \right) ,\text{ }u\text{ is a very
weak solution of }P(f,0,1)\text{ for some }f\in L_{\zeta }^{1}\left( \Omega
\right) \}.
\end{equation*}%
Then $A$ is a $m$-$T$-accretive operator on the Banach space $X$ and $%
\overline{D(A)}=X.$
\end{lemma}

Thus the Crandall-Liggett theorem can be applied to get the existence and
uniqueness of $u\in \mathcal{C}\left( \left[ 0,T\right] ;L_{\delta
}^{1}\left( \Omega \right) \right) $ $L_{\delta }^{1}$\emph{-mild solution}
of \eqref{eq nonhomogeneous P(1)}. Moreover, $u$ is a very weak solution of %
\eqref{eq nonhomogeneous P(1)} in the sense that $u\in \mathcal{C}\left( %
\left[ 0,T\right] ;L_{\delta }^{1}\left( \Omega \right) \right) ,$ $u\geq
0,u^{m}\in L^{1}\left( \left( 0,T\right) \times \Omega \right) ,f={u^{-\beta
}}{\chi _{\left\{ {u>0}\right\} }}\in {L^{1}}\left( {{0,T:}L_{\delta
}^{1}\left( \Omega \right) }\right) $ and for any $t\in \lbrack 0,T]$ 
\begin{eqnarray*}
&&\int_{\Omega }{u\left( t,x\right) \zeta (x)dx}+\int_{0}^{t}{\int_{\Omega }{%
{{u(t,x)^{m}}}dxdt}} \\
&=&\int_{\Omega }{{u_{0}(x)}\zeta (x)dx+}\int_{0}^{t}{\int_{\Omega }}f(t,x){%
\zeta (x){dxdt-}\int_{0}^{t}\int_{\partial \Omega }}\frac{\partial \psi }{%
\partial n}{(x){dx}}.
\end{eqnarray*}

The rest of arguments is completely similar to the case of problem \eqref{P}.

\bigskip

Now, let us present some explicit examples of solution of \eqref{P(1)}:

\begin{lemma}
\label{Lemma explicit solution} i) Let $q\in (-\infty ,1),$ $x_{0}\in 
\mathbb{R}^{N},$ and for $C>0$ define the function 
\begin{equation}
v_{q,C}(x)=C\left\vert x-x_{0}\right\vert ^{\frac{2}{1-q}}.  \label{Def v_C}
\end{equation}%
Then, for any $\lambda >0$%
\begin{equation}
\mathcal{L}(v):=-\Delta {v}+\lambda v{^{q}}=\left[ \lambda C^{2}-\frac{%
2(N(1-q)+2q)}{(1-q)^{2}}C\right] \left\vert x-x_{0}\right\vert ^{\frac{2q}{%
1-q}}.  \label{Oper L}
\end{equation}%
In particular, if we define 
\begin{equation}
K_{N,q,\lambda }=\left[ \frac{\lambda (1-q)^{2}}{2(N(1-q))+2q}\right] ^{%
\frac{1}{1+\beta /m}},  \label{Constant N, q}
\end{equation}%
then $\mathcal{L}(v)\equiv 0$ if $C=K_{N,q,\lambda }$ and $\mathcal{L}(v)>0$
(resp. $\mathcal{L}(v)<0)$ if $C<K_{N,q,\lambda }$ (resp. $C>K_{N,q,\lambda
}).$

\noindent ii) If for $m>0$ and $\beta \in (0,m)$ we define 
\begin{equation*}
u_{\beta ,m,C}(x)=(v_{q,C}(x))^{1/m}\text{= }C^{1/m}\left\vert
x-x_{0}\right\vert ^{\frac{2}{m+\beta }},\text{ i.e. with }q=-\beta /m,
\end{equation*}%
then 
\begin{equation}
-\Delta (u_{\beta ,m,C})^{m}+\lambda (u_{\beta ,m,C}){^{-\beta }}=\left[
\lambda C^{2}-\frac{2m(N(m+\beta )-2\beta )}{(m+\beta )^{2}}C\right]
\left\vert x-x_{0}\right\vert ^{\frac{-2\beta }{m+\beta }}.
\label{Oper Porous media}
\end{equation}%
\noindent ii.a) Define 
\begin{equation}
K_{N,m,\beta ,\lambda }=\left[ \frac{\lambda (m+\beta )^{2}}{2m(N(m+\beta
)-2\beta )}\right] ^{\frac{m}{m+\beta }},  \label{K_m_beta}
\end{equation}%
then $K_{N,m,\beta ,\lambda }>0$ and $-\Delta (u_{\beta ,m,C})^{m}+\lambda
(u_{\beta ,m,C}){^{-\beta }=0}$ in $\mathbb{R}^{N}$ if $C=K_{N,q,\lambda }.$

\noindent ii.b) If $x_{0}\in \overline{\Omega }$ then $-\Delta (u_{\beta
,m,C})^{m}+\lambda (u_{\beta ,m,C}){^{-\beta }\in L_{\delta }^{1}\left(
\Omega \right) }$ and $-\Delta (u_{\beta ,m,C})^{m}+\lambda (u_{\beta ,m,C}){%
^{-\beta }}>0$ (resp. $<0)$ if $C<K_{N,q,\lambda }$ (resp. $C>K_{N,q,\lambda
}).$

\noindent iii) If $m>0$ and $\beta \in \lbrack m,+\infty )$ then 
\eqref{Oper Porous
media} holds in $\mathbb{R}^{N}$. Moreover, the constant given by (\ref%
{K_m_beta}) is such $K_{N,m,\beta ,\lambda }>0$ if and only if $N\geq 2.$

\noindent iii.a) If $x_{0}\in \partial \Omega $ and $\delta (x)=\left\vert
x-x_{0}\right\vert $ then $-\Delta (u_{\beta ,m,C})^{m}+\lambda (u_{\beta
,m,C}){^{-\beta }\in L_{\delta }^{1}\left( \Omega \right) }$ and $-\Delta
(u_{\beta ,m,C})^{m}+\lambda (u_{\beta ,m,C}){^{-\beta }}>0$ (resp. $<0)$ if 
$C<K_{N,q,\lambda }$ (resp. $C>K_{N,q,\lambda }).$

\noindent iii.b) If $x_{0}\in \Omega $ then $-\Delta (u_{\beta
,m,C})^{m}+\lambda (u_{\beta ,m,C}){^{-\beta }\notin L_{\delta }^{1}\left(
\Omega \right) .}$
\end{lemma}

\noindent \textit{Proof.} Part i) was given in Lemma 1.6 of \cite{Diaz1985}.
Part ii) result from i) by a simple change of variable. Moreover, the fact
that $-\Delta (u_{\beta ,m,C})^{m}+\lambda (u_{\beta ,m,C}){^{-\beta }\in
L_{\delta }^{1}\left( \Omega \right) }$ holds because 
\begin{equation}
\frac{-2\beta }{m+\beta }+1>-1,  \label{Inequality in L^1 distance}
\end{equation}%
for the case $x_{0}\in \partial \Omega $ and since 
\begin{equation}
\frac{-2\beta }{m+\beta }>-1,  \label{Inqulity in L^}
\end{equation}%
(thanks to the condition $\beta \in (0,m)$) when $x_{0}\in \Omega $. From
the definition \eqref{K_m_beta} we see that if $\beta \in \lbrack m,+\infty
) $ then the positivity of $K_{N,m,\beta ,\lambda }$ fails only for $N=1.$
Moreover, inequality \eqref{Inequality in L^1 distance} still holds true,
but we see that for any interior point $x_{0}\in \Omega $ the weight $\delta
(x)$ is not from any help and thus the singularity is not integrable (since
condition \eqref{Inqulity in
L^} fails if $\beta \geq m$). $_{\square }$

\bigskip

\noindent \textbf{Corollary 1.\label{Corollary sharp} }\textit{Let }$\Omega
=B_{R}(x_{0})$\textit{\ and take }$u_{0}(x)=u_{\beta ,m,C}(x)$\textit{\ with 
}$C=K_{N,m,\beta ,\lambda }$\textit{\ and }$\lambda =1.$ \textit{Let }$R>0$%
\textit{\ be such that }$R^{\frac{2m}{m+\beta }}=1.$\textit{\ Then }$%
u(t,x)=u_{\beta ,m,C}(x)$\textit{\ is the unique solution of \eqref{P(1)}.
Moreover }%
\begin{equation*}
{\left\Vert {\nabla {u^{\frac{m+\beta }{2}}}\left( t\right) }\right\Vert _{{%
L^{\infty }}\left( \Omega \right) }=C}^{\ast },
\end{equation*}%
\textit{for some }$C^{\ast }>0$\textit{\ \ and the exponent }$\frac{m+\beta 
}{2}$\textit{\ cannot be replaced by any other greater exponent }$\alpha $%
\textit{\ such that }$\left\Vert {\nabla {u^{\alpha }}\left( t\right) }%
\right\Vert _{{L^{\infty }}\left( \Omega \right) }<+\infty $\textit{.}$%
_{\square }$

\bigskip

In order to prove some other qualitative properties it is useful the
following result:

\bigskip

\begin{lemma}
\label{Lemma superr-subsolutions} i) Let $q\in (-\infty ,1),$ $x_{0}\in 
\mathbb{R}^{N},t_{0}\geq 0$ and for $C>0$ define the function 
\begin{equation}
v_{q,C}(x)=C\left\vert x-x_{0}\right\vert ^{\frac{2}{1-q}}.
\end{equation}%
Given $t_{0}\geq 0,\theta \geq 0$ and $\lambda >0,$ let 
\begin{equation*}
y_{q,\theta ,\lambda }(t)=\left[ \theta ^{1-q}-\lambda (1-q)(t-t_{0})\right]
_{+}^{\frac{1}{1-q}},\text{ for }t\geq t_{0},
\end{equation*}%
so that 
\begin{equation*}
y_{q,\theta ,\lambda }(t)=0\text{ for any }t\geq \frac{\theta ^{1-q}}{%
\lambda (1-q)}.
\end{equation*}%
Then, given $m\geq 1,$ if $C\leq K_{N,q,\lambda }$, the function 
\begin{equation}
U(t,x)=\left[ v_{q,C}(x)+y_{q,\theta ,\lambda }(t)^{m}\right] ^{\frac{1}{m}},
\label{Supersolution U}
\end{equation}%
satisfies 
\begin{equation*}
{\partial _{t}}U-\Delta {U^{m}+\mu }U^{q}\geq 0\text{ on }(t_{0},+\infty
)\times \mathbb{R}^{N},
\end{equation*}%
with $\mu =2\lambda .$

\noindent ii) If for $m\geq 1$ and $\beta \in (0,m),$ we define%
\begin{equation*}
z_{m,\beta ,\theta ,\lambda }(t)=\left[ \theta ^{\frac{m+\beta }{m}}-\lambda
(\frac{m+\beta }{m})(t-t_{0})\right] _{+}^{\frac{m}{m+\beta }},\text{ for }%
t\geq t_{0}.
\end{equation*}%
and thus 
\begin{equation*}
W(t,x)=\left[ u_{\beta ,m,C}(x)^{m}+z_{m,\beta ,\theta ,\lambda }(t)^{m}%
\right] ^{\frac{1}{m}},
\end{equation*}

\noindent then, if $\lambda =\frac{1}{2}$ and $C\geq K_{N,q,\lambda },$ we
have%
\begin{equation*}
{\partial _{t}}W-\Delta W{^{m}+}W^{-\beta }{\chi _{\left\{ W{>0}\right\} }}%
\leq 0\text{ on }(t_{0},+\infty )\times \mathbb{R}^{N}.
\end{equation*}
\end{lemma}

\noindent \textit{Proof. }Notice that

\begin{equation*}
\left\{ 
\begin{array}{lc}
\frac{dy_{q,\theta ,\lambda }}{dt}+\lambda y_{q,\theta ,\lambda }^{q}=0 & 
\\ 
y_{q,\theta ,\lambda }(t_{0})=\theta . & 
\end{array}%
\right.
\end{equation*}%
Moreover, from the convexity of the function $s\rightarrow s^{m}$ we get
that 
\begin{equation*}
{\partial _{t}}U=U^{-\frac{m-1}{m}y_{q,\theta ,\lambda }^{m-1}}\frac{%
dy_{q,\theta ,\lambda }}{dt}\geq \frac{dy_{q,\theta ,\lambda }}{dt},
\end{equation*}%
moreover

\begin{equation*}
-\Delta {U^{m}=}-\Delta v_{q,C}.
\end{equation*}%
Notice also that

\begin{equation*}
(a+b)^{r}\geq \frac{a^{r}+b^{r}}{2}\text{, for any }a,b\geq 0\text{ and }r>0.
\end{equation*}%
Then 
\begin{eqnarray*}
{\partial _{t}}U-\Delta {U^{m}+\mu }U^{q} &\geq &\frac{dy_{q,\theta ,\lambda
}}{dt}-\Delta v_{q,C}+2{\lambda }\left[ v_{q,C}(x)+y_{q,\theta ,\lambda
}(t)^{m}\right] ^{\frac{q}{m}} \\
&\geq &\left( \frac{dy_{q,\theta ,\lambda }}{dt}+{\lambda }y_{q,\theta
,\lambda }^{q}\right) -\Delta v_{q,C}+{\lambda }v{^{q}}\geq 0.
\end{eqnarray*}%
The proof of ii) is similar but uses now that 
\begin{equation*}
(a+b)^{-r}\leq \frac{a^{-r}+b^{-r}}{2}\text{, for any }a,b>0\text{ and }%
r>0._{\square }
\end{equation*}

\bigskip

Here are some applications of the above Lemma.

\begin{proposition}
\label{Proposition Finite speed} Let $m\geq 1$, $\beta \in (0,m)$ and
consider $u_{0}\in L^{\infty }(\Omega )$, $u_{0}\geq 0$. Then:

\noindent i) Complete quenching and formation of the free boundary: there is
a finite time $\tau _{0}>0$ such that if $u$ is the mild solution of %
\eqref{P} 
\begin{equation*}
u\left( t,x\right) =0,\ \forall t\in \left( \tau _{0},\infty \right) 
\mbox{ and
a.e. }x\in \Omega .
\end{equation*}

\noindent ii) Let $m\geq 1$, $\beta \in (0,m)$. Assume (for simplicity) $%
1\geq u_{0}\geq 0$. If $u$ is the mild solution of \eqref{P(1)} then for
a.e. $x_{0}\in \Omega $ such that $\delta (x_{0})=d(x_{0},\partial \Omega
)\geq $ $\left( K_{N,q,\lambda }\right) ^{-\frac{2}{1-q}}$ there exists a $%
\tau _{0}=\tau _{0}(x_{0})\geq 0$ such that 
\begin{equation}
u\left( t,x_{0}\right) =0,\ \forall t\in \left( \tau _{0},\infty \right) .
\label{local quenching}
\end{equation}

\noindent iii) Let $m\geq 1$, $\beta \in (0,m).$If 
\begin{equation*}
0\leq {u_{0}(x)\leq }K_{N,q,\lambda }\left\vert x-x_{0}\right\vert ^{\frac{2%
}{1-q}}\text{ a.e. on }B_{\delta (x_{0})}(x_{0})\cap \Omega \text{ and }%
\delta (x_{0})\geq \frac{1}{\left( K_{N,q,\lambda }\right) ^{\frac{m+\beta }{%
2m}}}
\end{equation*}%
then, if $u$ is the mild solution of \eqref{P} we get that%
\begin{equation*}
0\leq {u(t,x)\leq }K_{N,q,\lambda }\left\vert x-x_{0}\right\vert ^{\frac{2}{%
1-q}}\text{ a.e. on }(0,+\infty )\times B_{\delta (x_{0})}(x_{0})\cap \Omega
\end{equation*}%
and, in particular ${u(t,x}_{0}{)=0}$ for any $t>0.$

\noindent iv) Let $m\geq 1$, $\beta \in (0,m)$ and assume 
\begin{equation}
{u_{0}(x)\geq }\left[ C\delta (x)^{\frac{2m}{m+\beta }}+\theta ^{m}\right] ^{%
\frac{1}{m}},\text{ }\delta (x)=d(x,\partial \Omega )
\label{Initial data big}
\end{equation}%
for some $C\geq K_{N,q,\lambda }$. Then if\ $u$ is the mild solution of (\ref%
{P1}) and $\theta \leq 1$ we have%
\begin{equation*}
{u(t,x)\geq }W(t,x)\text{ for any }x\in \Omega \text{ and any }t>0.\text{ }
\end{equation*}%
In particular, if $\theta >0$ then 
\begin{equation*}
{u(t,x)>0}\text{ for any }x\in \Omega \text{ \ and }t\in \lbrack 0,\frac{%
2m\theta ^{\frac{m+\beta }{m}}}{m+\beta }).
\end{equation*}%
The conclusion holds for solutions of \eqref{P}, for any $x\in \Omega $ and $%
t>0$ if in the assumtion \eqref{Initial data big} we take $\theta =0$.
\end{proposition}

\noindent \textit{Proof.} i) Let $M={\left\Vert {u_{0}}\right\Vert _{{%
L^{\infty }}\left( \Omega \right) }}$. Notice that since $u^{-\beta } \geq
\mu u^{\alpha }$ for any $u\in (0,M]$ and any $q\in (0,1)$ if $0\leq \mu
\leq M^{-(\alpha +\beta )},$ then 
\begin{equation}
0\leq u\left( t,x\right) \leq U_{q}\left( t,x\right) ,\mbox{ a.e. in }\left(
0,T\right) \times \Omega ,
\end{equation}%
with $U_{q}$ the unique mild solution of the porous media homogeneous
problem with a possible \textit{strong absorption} 
\begin{equation}
\left\{ 
\begin{split}
& {\partial _{t}}U-\Delta {U^{m}+}\lambda U^{q}=0,\mbox{ in }\left( {0,T}%
\right) \times \Omega , \\
& U^{m}=0,\mbox{ on }\left( {0,T}\right) \times \partial \Omega , \\
& U\left( {0,x}\right) ={u_{0}}\left( x\right) ,\mbox{ in }\Omega ,
\end{split}%
\right.
\end{equation}%
since we know that $0\leq {u(t,x)}\leq M.$ Then if $U$ is given by (\ref%
{Supersolution U}) we get that

\begin{equation*}
0\leq U_{q}\left( t,x\right) \leq U\left( t,x\right) \text{ on }(0,+\infty
)\times \Omega
\end{equation*}%
if we take $t_{0}=0$ and $\theta \geq M$ (remember that $v_{q,C}(x)\geq 0$).
Taking $x_{0}$ (in the definition of \eqref{Def v_C}) arbitrary in $\mathbb{R%
}^{N}$ we get the conclusion.

\noindent ii) We argue as in i) and thus 
\begin{equation}
0\leq u\left( t,x\right) \leq U_{q}\left( t,x\right) ,\mbox{ a.e. in }\left(
0,T\right) \times \Omega ,
\end{equation}%
but now with $U_{q}$ the unique mild solution of the problem\textit{\ }%
\begin{equation}
\left\{ 
\begin{split}
& {\partial _{t}}U-\Delta {U^{m}+}\lambda U^{q}=0,\mbox{ in }\left( {0,T}%
\right) \times \Omega , \\
& U^{m}=1,\mbox{ on }\left( {0,T}\right) \times \partial \Omega , \\
& U\left( {0,x}\right) ={u_{0}}\left( x\right) ,\mbox{ in }\Omega ,
\end{split}%
\right.
\end{equation}%
We use the function $U$ given by \eqref{Supersolution U} as supersolution
and we conclude that if we take $t_{0}=0$ and $\theta \geq M$ and $x_{0}\in
\Omega $ such that $\delta (x_{0})=d(x_{0},\partial \Omega )\geq $ $\left(
K_{N,q,\lambda }\right) ^{-\frac{2}{1-q}}$ then (since $y_{q,\theta ,\lambda
}(t)\geq 0$) 
\begin{equation*}
U_{q}^{m}\left( t,x\right) \leq 1\leq C\delta (x_{0})^{\frac{2}{1-q}}\leq
v_{q,C}(x)\leq U^{m}(t,x)\text{ for }x\in \partial B_{\delta (x_{0})}(x_{0})
\end{equation*}%
and thus

\begin{equation*}
0\leq U_{q}\left( t,x\right) \leq U\left( t,x\right) \text{ on }(0,+\infty
)\times B_{\delta (x_{0})}(x_{0})
\end{equation*}%
if we take $t_{0}=0$ and $\theta \geq $ ${\left\Vert {u_{0}}\right\Vert _{{%
L^{\infty }}\left( B_{\delta (x_{0})}(x_{0})\right) }}$, which proves (\ref%
{local quenching}).

\noindent The proof of iii) is similar to to the proof of ii) but even
simpler than before since now $u=0$ on the boundary and the supersolution is
nonnegative.

\noindent The comparison of solutions $u$ of \eqref{P1} (respectively %
\eqref{P} with the subsolution $W(t,x)$ uses some properties of the function 
$\delta \left( x\right) =d\left( x,\partial \Omega \right) $ and follows the
same arguments than \cite{Davila2005} (see also \cite{Diaz-Giacomoni} and
Theorem 2.3 of \cite{Alvarez-Diaz}) thanks to the assumption $\beta \leq m$.$%
_{\square }$

\bigskip

\begin{remark}
Conclusion iv) of Proposition \ref{Proposition Finite speed} is very useful
in order to prove the uniqueness of the very weak solution of \eqref{P}
(see, e.g. \cite{Davila2005} and \cite{Diaz-Giacomoni}).
\end{remark}

\bigskip

A sharper estimate on the complete quenching time can be obtained without
passing by the porous media homogeneous problem with a possible \textit{%
strong absorption. }

\begin{proposition}
\label{theorem2 copy(1)} Assume the same conditions of Theorem \ref{theorem1}%
, part i). Then, every weak solution of equation \eqref{P} must vanish after
a finite time, i.e., there is a finite time $\tau _{0}>0$ such that 
\begin{equation*}
u\left( t,x\right) =0,\ \forall t\in \left( \tau _{0},\infty \right) 
\mbox{ and
a.e. }x\in \Omega .
\end{equation*}
\end{proposition}

\noindent \textit{Proof. } By Theorem \ref{theorem1}, it suffices to show
that the maximal solution $u$ constructed in the above Section vanishes
after a finite time $\tau _{0}>0$. Thanks to the smoothing effect we can
assume without loss of generality that the initial datum is a nonnegative
bounded function $u_{0}\in L^{\infty }\left( \Omega \right) .$We shall use
some energy methods in the spirit of (\cite{Antontsev2002} and \cite[Theorem
3]{Dao2016}). For any $q\geq \beta +2$, we can use $u^{q-1}$ as a test
function to equation \eqref{P} and we obtain 
\begin{equation*}
\frac{1}{q}\frac{d}{{dt}}\int_{\Omega }{{u^{q}}\left( {t,x}\right) dx}+\frac{%
{4m\left( {q-1}\right) }}{{{{\left( {m+q-1}\right) }^{2}}}}\int_{\Omega }{{{%
\left\vert {\nabla {u^{\left( {m+q-1}\right) /2}}\left( {t,x}\right) }%
\right\vert }^{2}}dx}+\int_{\Omega }{{u^{q-\beta -1}}\left( {t,x}\right) dx}%
=0.
\end{equation*}%
Define $v:={u^{\left( {m+q-1}\right) /2}}$. By applying the Sobolev
embedding to $v$, one obtains 
\begin{equation}
{\left\Vert {v\left( t\right) }\right\Vert _{{L^{{2^{\star }}}}\left( \Omega
\right) }}\leq C\left( N\right) {\left\Vert {\nabla v\left( t\right) }%
\right\Vert _{{L^{2}}\left( \Omega \right) }},  \label{4.2}
\end{equation}%
with 
\begin{equation*}
{2^{\star }}:=\left\{ 
\begin{split}
& \frac{{2N}}{{N-2}},\mbox{ if }N\geq 3, \\
& l,\mbox{ for }l\in \left( {1,\infty }\right) ,\mbox{ if }N=1,2.
\end{split}%
\right.
\end{equation*}%
As we shall see, it is enough to consider the case of $N\geq 3$ since the
cases of $N=1,2$ can be obtained by easy modifications. Observe that %
\eqref{4.2} is equivalent to 
\begin{equation*}
\left\Vert {u\left( t\right) }\right\Vert _{{L^{{q_{\star }}}}\left( \Omega
\right) }^{\frac{{{q_{\star }}\left( {N-2}\right) }}{N}}\leq C\left(
N\right) \int_{\Omega }{{{\left\vert {\nabla {u^{\left( {m+q-1}\right) /2}}%
\left( t,x\right) }\right\vert }^{2}}dx},
\end{equation*}%
with $q_{\star }:=\left( m+q-1\right) N/\left( N-2\right) $. Note that $%
q_{\star }>q$. From the interpolation inequality 
\begin{equation*}
{\left\Vert {u\left( t\right) }\right\Vert _{{L^{q}}\left( \Omega \right) }}%
\leq \left\Vert {u\left( t\right) }\right\Vert _{{L^{{q_{\star }}}}\left(
\Omega \right) }^{\theta }\left\Vert {u\left( t\right) }\right\Vert _{{%
L^{q-\beta -1}}\left( \Omega \right) }^{1-\theta },
\end{equation*}%
with $1/q=\theta /q_{\star }+\left( 1-\theta \right) /\left( q-\beta
-1\right) $, by a combination of the above inequalities, we deduce 
\begin{align*}
\left\Vert {u\left( t\right) }\right\Vert _{{L^{q}}\left( \Omega \right) }^{%
\frac{{{q_{\star }}\left( {N-2}\right) }}{N}}& \leq C\left\Vert {\nabla {%
u^{\left( {m+q-1}\right) /2}}}\right\Vert _{{L^{2}}\left( \Omega \right)
}^{2\theta }\left\Vert {u\left( t\right) }\right\Vert _{{L^{q-\beta -1}}%
\left( \Omega \right) }^{\frac{{\left( {1-\theta }\right) {q_{\star }}\left( 
{N-2}\right) }}{N}} \\
& \leq C{A^{\theta }}{A^{\frac{{\left( {1-\theta }\right) {q_{\star }}\left( 
{N-2}\right) }}{{\left( {q-\beta -1}\right) N}}}}=C{A^{\theta +\frac{{\left( 
{1-\theta }\right) {q_{\star }}\left( {N-2}\right) }}{{\left( {q-\beta -1}%
\right) N}}}},
\end{align*}%
where 
\begin{equation*}
A:=\int_{\Omega }{{{\left\vert {\nabla {u^{\left( {m+q-1}\right) /2}}\left( {%
t,x}\right) }\right\vert }^{2}}dx}+\int_{\Omega }{{u^{q-\beta -1}}\left( {t,x%
}\right) dx}.
\end{equation*}%
This implies 
\begin{equation*}
{\left\Vert {u\left( t\right) }\right\Vert _{{L^{q}}\left( \Omega \right) }}%
\leq C\left( {N,m,q}\right) {A^{\frac{\theta }{{{q_{\star }}}}\frac{N}{{N-2}}%
+\frac{{1-\theta }}{{q-1-\beta }}}}\leq C{A^{\frac{1}{q}+\frac{{2\theta }}{{%
\left( {N-2}\right) {q_{\star }}}}}}.
\end{equation*}%
Then 
\begin{equation*}
\frac{1}{q}\frac{d}{{dt}}\int_{\Omega }{{u^{q}}\left( {t,x}\right) dx}%
+C\left( {m,q}\right) A\leq 0.
\end{equation*}%
In particular, we obtain that $y\left( t\right) :=\left\Vert {u\left(
t\right) }\right\Vert _{{L^{q}}\left( \Omega \right) }^{q}$ satisfies\ the
following ordinary differential inequality 
\begin{equation}
y^{\prime }\left( t\right) +C{y^{\sigma }}\left( t\right) \leq 0,
\label{4.7}
\end{equation}%
with $\sigma :={\left( {1+2q\theta /\left( {\left( {N-2}\right) {q_{\star }}}%
\right) }\right) ^{-1}}\in \left( {0,1}\right) $. Then, as in (\cite%
{Antontsev2002}) we deduce that there is a time $\tau _{0}>0$ such that $%
y\left( \tau _{0}\right) =0$ and then $y\left( t\right) =0$ for any $t>\tau
_{0}$ since $y\left( t\right) $ is a non-negative function. Thus, $u\left(
t,x\right) =0,\mbox{ in }\left( \tau _{0},\infty \right) \times \Omega .$
Indeed, if on the contrary we assume that $y\left( t\right) >0$ for every $%
t>0$ then by solving \eqref{4.7}, we get that ${y^{1-\sigma }}\left(
t\right) +Ct\leq {y^{1-\sigma }}\left( 0\right) .$ and since this inequality
holds for any $t>0$ we arrive to a contradiction for $t$ large enough. This
ends the proof.$_{\square }$

\begin{remark}
We note that the above arguments are independent of the size of $\Omega $.
Thus, one can easily verify that the quenching result also holds for the
case $\Omega =\mathbb{R}^{N}$ as pointed out in the Introduction. Moreover
the formation of the free boundary given in Proposition \ref{Proposition
Finite speed} can be also adapted to solutions of the Cauchy problem.
\end{remark}

\begin{remark}
Although several energy methods were developed in the literature (see, e.g., 
\cite{Antontsev2002, Diaz1985}, and their references) the main new aspect
was the application to the case of singular absorption terms. The method
applies to the class of \textit{local} weak solutions of the more general
formulation 
\begin{equation}
\frac{\partial \psi \left( v\right) }{\partial t}-\mathrm{div}\,\mathbf{A}%
\left( x,t,v,Dv\right) +B\left( x,t,v,Dv\right) +C\left( x,t,v\right)
=f\left( x,t,u\right) ,  \label{Ecuation General quasilinear}
\end{equation}%
in which the absorption term can be singular and then including equation %
\eqref{P} as a special case. More precisely the assumptions made in \cite%
{Diaz2014} were the following: under the general structural assumptions%
\begin{equation*}
\begin{array}{c}
|\mathbf{A}(x,t,r,\mathbf{q})|\leq C|\mathbf{q}|,C|\mathbf{q}|^{2}\leq 
\mathbf{A}(x,t,r,\mathbf{q})\cdot \mathbf{q,} \\ 
C|r|^{\theta +1}\leq G(r)\leq C^{\ast }|r|^{\theta +1},\text{ }%
\end{array}%
\end{equation*}%
where%
\begin{equation*}
G(r)=\psi (r)\,r-\int_{0}^{r}\psi (\tau )\,d\tau ,
\end{equation*}%
and 
\begin{equation*}
C|r|^{\alpha }\leq C(x,t,r)\,r,
\end{equation*}%
\begin{equation}
f(x,t,r)r\leq \lambda |r|^{q+1}+g(x,t)r\text{, }  \label{H-f}
\end{equation}%
with $p>1,q\in \mathbb{R}$ and the main assumptions 
\begin{equation}
\theta \in (0,1),  \label{Gamma}
\end{equation}%
and $\alpha \in (0,\min (1,2\theta )).$Notice that by defining $v=u^{m}$
(and thus $u=v^{1/m}$), problem \eqref{P} can be formulated as 
\begin{equation}
\left\{ 
\begin{split}
& {\partial _{t}}v^{1/m}-\Delta {v}+v{^{-\beta /m}}{\chi _{\left\{ v{>0}%
\right\} }}=0,\mbox{ in }\left( {0,\infty }\right) \times \Omega , \\
& v=0,\mbox{ on }\left( {0,\infty }\right) \times \partial \Omega , \\
& v\left( {0,x}\right) ={u_{0}^{1/m}}\left( x\right) ,\mbox{ in }\Omega .
\end{split}%
\right.  \label{Problem in v}
\end{equation}%
Thus, it corresponds to equation \eqref{Ecuation General quasilinear} with 
\begin{equation*}
\mathbf{A}\left( x,t,v,Dv\right) =Dv,\text{ }B\left( x,t,v,Dv\right)
=0,f\left( x,t,u\right) =0\text{ }
\end{equation*}%
$C(x,t,r)=v{^{-\beta /m}}{\chi _{\left\{ \ \ v{>0}\right\} \ \ }}$and $\psi
\left( v\right) =v^{1/m}.$ Then the corresponding exponents are $\theta =1/m$%
, $\alpha =\frac{m-\beta }{m}$ and the energy method apply presented in \cite%
{Diaz2014} applies to the cases: 
\begin{equation*}
\left\{ 
\begin{array}{lr}
\beta \in (0,m) & \text{if }m\in \lbrack 1,2] \\ 
\beta \in (m-2,m) & \text{if }m>2.%
\end{array}%
\right.
\end{equation*}%
Theorem 1 of \cite{Diaz2014} shows the finite speed of propagation, and more
exactly a stronger property which usually is as called "stable (or uniform)
localization property" (see also \cite{Antontsev2002}, Chapter 3). A
sufficient condition for the existence of \textit{local waiting time} (or,
what we can call perhaps more properly as the \textit{non dilation of the
initial support)}: the free boundary cannot invade the subset where the
initial datum is nonzero was given in Theorem 3 of \cite{Diaz2014}. Finally,
the local quenching property (i.e. the formation of a \textit{region where }$%
u=0$ even for strictly positive initial data: sometimes called also as 
\textit{the instantaneous shrinking of the support property}: see \cite%
{Antontsev2002} and its references) was shown in Theorem 4 of \cite{Diaz2014}%
.
\end{remark}

\begin{remark}
\label{Remark b>m}Let us recall that in the case of the semilinear
formulation of problem \eqref{Lemma P(1)}, with $\beta \geq 1$ it is known
that there is a finite time blow up $\tau _{0}$ of the time derivative ${%
\partial _{t}}u$ in the interior points $x_{0}\in \Omega $ where the
solution quenches ($u(\tau _{0},x_{0})=0$) and that weak solutions ceases to
exits for $t>\tau _{0}$ (see, e.g., the exposition made in \cite%
{Kawarada1974/75}, \cite{Kawohl1992a}, \cite{Levine1993} and \cite%
{Filippas-Guo} \cite{Guo 2007}). Nevertheless, it is possible to show that
in the case in which the singularity is automatically present on the
boundary of $\Omega $ from the initial time $t=0,$ the existence of a very
weak solution can be obtained at least until the time in which the solution
also quenches in some interior point $x_{0}\in \Omega $. The mean reason of
this fact is that the weight $\delta (x)=d(x,\partial \Omega )$ used in the
definition of very weak solution, when asking that ${u^{-\beta }}{\chi
_{\left\{ {u>0}\right\} }}\in {L^{1}}\left( {{0,T:}L_{\delta }^{1}\left(
\Omega \right) }\right) $, allows to compensate the singularity arising in
the boundary (but obviously it is ineffective for singularities arising in
the inerior of the domain $\Omega $). In fact the above compensation of the
boundary singularity, when $\beta \geq m$, with the weight $\delta (x)$ was
already pointed out in parts iii.a) and iii.b) of Lemma \ref{Lemma explicit
solution}. A global example which requires some additional assumptions and
holds for a modified equation 
\begin{equation*}
{\partial _{t}}u-\Delta {u^{m}}+\lambda \delta (x)^{\nu }{u^{-\beta }}{\chi
_{\left\{ \ \ {u>0}\right\} \ \ }}=0,\mbox{ in }\left( {0,\infty }\right)
\times \Omega
\end{equation*}%
for some suitable values of $\lambda >0$ and $\nu >1.$ This corresponds to
an easy adaptation to the framework of the slow diffusion with a singular
term some of the results announced in \cite{DHRako} and Section 7 of \cite%
{Rakotoson2011} concerning the associate semilinear problems.
\end{remark}

\textbf{Acknowledgements}

The research of J.I. D\'{\i}az was partially supported by the project ref.
MTM2017-85449-P of the Ministerio de Ciencia, Innovaci\'{o}n y Universidades
-- Agencia Estatal de Investigaci\'{o}n (Spain).


\end{document}